\documentclass{amsart}
\usepackage{amsmath}
\usepackage{amsthm}
\usepackage{amssymb}
\usepackage{amsfonts}

\usepackage[utf8]{inputenc}
\usepackage[T1]{fontenc}

\usepackage{xspace}
\usepackage[dvipsnames]{xcolor}
\usepackage{tikz}
\usepackage{graphicx}
\usepackage{mathtools}
\usepackage{hyperref}
\usepackage{nameref} 
\usepackage{geometry,pdflscape}

\usepackage[
    backend=biber,
]{biblatex}
\newtheorem{theorem}{Theorem}[section]
\newtheorem{lemma}[theorem]{Lemma}

\theoremstyle{definition}
\newtheorem{definition}[theorem]{Definition}
\newtheorem{example}[theorem]{Example}

\theoremstyle{remark}
\newtheorem{remark}[theorem]{Remark}
\newtheorem{corollary}[theorem]{Corollary}
\newtheorem{notation}[theorem]{Notation}
\newtheorem{intuition}[theorem]{Intuition}

\numberwithin{equation}{section}

\newcommand{\C}{\mathbb{C}} 
\newcommand{\R}{\mathbb{R}} 
 
\newcommand{\K}{\mathbb{K}} 
\newcommand{\N}{\mathbb{N}}

\newcommand{\eg}{\emph{e.g.}\xspace}
\newcommand{\ie}{\emph{i.e.}\xspace}
 
\newcommand{\RP}{\mathbb{R}\mathbb{P}}
\newcommand{\CP}{\mathbb{C}\mathbb{P}}
\newcommand{\CsP}{\mathbb{C^*}\mathbb{P}}
\newcommand{\Smz}{\setminus \{0\}}
\newcommand{\RsP}{\mathbb{R^*}\mathbb{P}}

\newcommand{\bbar}{\; | \;}
\newcommand{\Ks}{\mathbb{K^*}}
\newcommand{\Rs}{\mathbb{R^*}}
\newcommand{\Cs}{\mathbb{C^*}}
\newcommand{\As}{{A^*}}

\newcommand{\Lim}{\mathbb{L}} 
\newcommand{\linf}{l_\infty}
\DeclareMathOperator{\hal}{\mathbf{hal}}
\DeclareMathOperator{\gal}{\mathbf{gal}}
\DeclareMathOperator{\inter}{\mathbf{int}}
\DeclareMathOperator{\sh}{{\mathbf{sh}}}
\newcommand*\conj[1]{\overline{#1}}
\newcommand{\norm}[1]{\left\lVert#1\right\rVert}
\newcommand{\I}{\mathbb{I}} 
\newcommand{\e}{\mathrm{e}}
\newcommand{\ii}{\mathrm{i}}
\newcommand{\II}{\mathrm{I}}
\newcommand{\JJ}{\mathrm{J}}
\DeclareMathOperator{\magni}{\mathbf{mag}}
\newcommand{\KsP}{\mathbb{K^*}\mathbb{P}}

\newcommand{\KP}{\mathbb{K}\mathbb{P}}
\DeclareMathOperator{\psh}{\mathbf{psh}}
\DeclarePairedDelimiter\abs{\lvert}{\rvert}
\newcommand{\CsZ}{\mathbb{C^*}\setminus \{0\}}
\makeatletter
    \def\page{p.\ }
    \def\Dx{{\Delta x}}

\makeatother

\DeclareMathOperator*{\argmax}{arg\,max}

\newcommand{\A}{\mathbb{A}}

\newcommand{\join}{\textbf{join}} 
\newcommand{\meet}{\textbf{meet}} 

\DeclareMathOperator{\phal}{\mathbf{\mathbb P-hal}}

\newcommand{\normsymb}{\norm{\cdot}}

\DeclareMathOperator{\Id}{Id}

\newcommand{\PCs}{\mathcal{P^*}} 
\newcommand{\LCs}{\mathcal{L^*}} 

\newcommand\Tau{\mathrm{T}}

\begin{document}

\title{First Steps in Non-standard Projective Geometry}

\author{Michael Strobel}
\address{Technical University of Munich, Germany}
\curraddr{}
\email{strobel@ma.tum.de}
\thanks{}


\date{\today}

\dedicatory{}

\begin{abstract}
  In this article, we will introduce methods of non-standard analysis into projective geometry. Especially, we will analyze the properties of a projective space over a non-Archimedean field. Non-Archimedean fields contain numbers that are smaller than every real number: the so called ``infinitesimal'' numbers. The theory is well known from non-standard analysis. This enables us to define projective objects that deviate only infinitesimally (in an appropriate metric) from each other. And we show that in most cases operations involving such objects that deviate infinitesimally also experience only infinitesimal change. Another focus will be where this property does not hold true and show that this usually involves discontinuity or degeneration.

  Furthermore, we will explore common projective concepts like projective transformations, cross-ratios and conics in a non-standard setting.
\end{abstract}

\maketitle

\section{Some Basics}
\label{chap:nsa}
An important step towards rigorous treatment of infinitesimal quantities was the article ``\textsc{Eine Erweiterung der Infinitesimalrechnung}'': \cite{schmieden1958erweiterung}. Their construction had the drawback that it wasn't actually a field and therefore contains elements with no multiplicative inverse. Nevertheless, this article was a huge step towards the formal treatment of the infinitesimals.

Eventually a thorough treatment of non-standard analysis was given by \cite{robinson1961non} which overcame the drawbacks of the ansatz by Laugwitz and Schmieden and constructed a proper field with infinitesimal and unlimited members.

We will not admit how to construct the hyperreal or hypercomplex numbers and refer to the literature: \cite{goldblatt,robinson1961non}.

\begin{notation}[$\Rs$ and $\Cs$, \cite{goldblatt} \page 25]
  We denote by $\Rs$ and $\Cs$ the \textit{hyperreal} and the \textit{hypercomplex} numbers.
\end{notation}

\begin{notation}[Enlargement, \cite{goldblatt} \page 28]
  For a set $A \subset \R$ or $A \subset \C$ we denote by $\mathit{\As}$ the \textit{enlargement of $\mathit{A}$}. 
\end{notation}
\begin{remark}
  The enlargement operation adds all non-standard members to a set (\eg infinitesimal and unlimited numbers). We will not got into detail here and refer to \cite{goldblatt}.
\end{remark}
  \begin{intuition}[Extended Function, \cite{goldblatt} \page 30]
    Let $f: A \subset \C \rightarrow \C$ be a function. We denote by $f^*: \As \rightarrow \Cs$ the \textit{extended function $\mathit{f}$}. It holds true that $f(z) = f^*(z) \; \forall z \in A$.
  \end{intuition}
  \begin{intuition}[Transfer, \cite{goldblatt} \page 45]
    \label{nsa_basics:transfer}
      \textit{Universal transfer}: if a property holds true for all real (complex) numbers, then it holds true for all hyperreal (hypercomplex) numbers. 
      \textit{Existential transfer}: if there exists a hyperreal (hypercomplex) number satisfying a certain property, then there exists a real (complex) number with this property.
  \end{intuition}
  \begin{remark}
    We will not give a formal definition of the transfer principle and the star transform (transferring a logical statement from the non-standard world forth and back) and refer to \cite{goldblatt}.
  \end{remark}
\begin{definition}[$\Ks$ sets, partially \cite{goldblatt} \page 50]
    \begin{align*}
    \mathbb{I}_\R &:= \{\epsilon \in \Rs: \, |\epsilon| < |r| \quad \forall r \in \R \},\quad \textit{real infinitesimal numbers}\\
    \mathbb{I}_\C &:= \{\epsilon \in \Cs: \, |\epsilon| < |r| \quad \forall r \in \R \},\quad \textit{complex infinitesimal numbers}\\
        \mathbb{A}_\R&:= \{r^* \in \Rs: \exists n \in \N: \frac{1}{n} < |r^*| <
n\},\quad \textit{real appreciable numbers}\\
        \mathbb{A}_\C&:= \{r^* \in \Cs: \exists n \in \N: \frac{1}{n} < |r^*| <
n\},\quad \textit{complex appreciable numbers}\\
        \mathbb{R_\infty^+}&:= \{H \in \Rs: \, H > r \quad \forall r \in \R
\},\quad \textit{positive unlimited numbers}\\
        \mathbb{R_\infty^-}&:= \{H \in \Rs: \, H < r \quad \forall r \in \R
        \},\quad \textit{negative unlimited numbers}\\
        \mathbb{R_\infty}&:= \R_\infty^+ \cup \R_\infty^- ,\quad \textit{real unlimited numbers}\\
	\mathbb{C_\infty}&:= \Cs \setminus \{ \mathbb{I}_\C \cup \mathbb{A}_\C \}  ,\quad \textit{complex unlimited numbers}\\
        \mathbb{L}_\R&:= \Rs \setminus \R_\infty    ,\quad \textit{real limited numbers}\\
        \mathbb{L}_\C&:= \Cs \setminus \C_\infty    ,\quad \textit{complex limited numbers}
\end{align*}
\end{definition}
\begin{remark}
  We will drop the index $\R$ or $\C$ of $\mathbb{I}_\R, \mathbb{I}_\C, \mathbb{A}_\R,  \mathbb{L}_\R$ or $ \mathbb{L}_\C$ if the context admits. 
\end{remark}
\begin{theorem}[Arithmetics, partially {\cite{goldblatt} \page 50-51}]
    \label{nsa_basics:arith}
   Let $n\in \N$, $\epsilon, \delta$ be infinitesimal, $b,c$ appreciable and $H,K$
   unlimited hypercomplex numbers. Then it holds true:
   \begin{itemize}
       \item Sums
           \begin{itemize}
               \item $\epsilon + \delta$ is infinitesimal
               \item $b + \epsilon$ is appreciable
               \item $b + c$ is limited 
               \item $H + \epsilon$ and $H+b$ are unlimited
           \end{itemize}
       \item Additive inverse
           \begin{itemize}
               \item $-\epsilon$ is infinitesimal
               \item $-b$ is appreciable
               \item $-H$ is unlimited
           \end{itemize}
       \item Products 
           \begin{itemize}
               \item $\epsilon \cdot \delta$ and $\epsilon \cdot b$ are infinitesimal
               \item $b \cdot c$ is appreciable 
               \item $b \cdot H$ and $H \cdot K$ are unlimited 
           \end{itemize}
       \item Reciprocals
           \begin{itemize}
               \item  $\frac{1}{\epsilon}$ is unlimited if $\epsilon \neq 0$
               \item  $\frac{1}{b}$ is appreciable
               \item  $\frac{1}{H}$ is infinitesimal 
           \end{itemize}
       \item Quotiens 
           \begin{itemize}
               \item  $\frac{\epsilon}{b}$, $\frac{\epsilon}{H}$ and $\frac{b}{H}$
                   are infinitesimal
               \item  $\frac{b}{c}$ is appreciable
               \item  $\frac{b}{\epsilon}$, $\frac{H}{\epsilon}$ and
                   $\frac{H}{b}$ are unlimited
           \end{itemize}
       \item Real roots, let $\epsilon \in \I_\R, b \in \A_\R, H \in \R_\infty^+$
           \begin{itemize}
               \item If $ \epsilon > 0, \sqrt[n] \epsilon$ is infinitesimal
               \item If $ b > 0, \sqrt[n] b$ is appreciable
               \item If $ H > 0, \sqrt[n] H$ is unlimited
           \end{itemize}
       \item Indetermined forms
           \begin{itemize}
               \item $\frac{\epsilon}{\delta}, \frac{H}{K}, \epsilon \cdot H, H +
                   K$ are undetermined 
           \end{itemize}
   \end{itemize}
\end{theorem}
\begin{definition}[Infinitely Close and Limited Distance, partially \cite{goldblatt} \page 52]
    Define for $b,c \in \Ks$ the equivalence
    relation \[ b \simeq c \text{ if and only if }, b-c \in \mathbb I \] and we
    call $b$ is \textit{infinitely close} to $c$. Furthermore, we write\[b
        \sim c \text{ if and only if },  b-c \in \mathbb L \]
        and we say $b$ has \textit{limited distance} to $c$.
\end{definition}

\begin{definition}[Halo and Galaxy, partially \cite{goldblatt} \page 52]
     Define the \textit{halo} of an arbitrary $b \in \Ks$ as 
    \[
        \hal(b) := \{c \in \Ks : b \simeq c\}
    \]
    and the \textit{galaxy} of $b \in \Ks$ as 
    \[
        \gal(b) := \{c \in \Ks : b \sim c\}.
    \]
\end{definition}
\begin{theorem}[Shadow]
    \label{nsa_basics:shadow}
    Every limited hyperreal (hypercomplex) $z^*$ is infinitely close to exactly one real (complex) number. We
    call this the \textit{shadow} of $z^*$ denoted by $\sh(z^*)$.
\end{theorem}
\begin{proof}
    The real case is proven in \cite{goldblatt} \page 53. We will prove the complex case using the constructions of the real case for both real and imaginary part. Then we use an estimate to conclude.

    Write $z^* = a + \ii \cdot b$, then we have $a,b \in \Lim_\R$ (otherwise $z^*$ would not be limited). Define the following sets:
    \[
        A:=\{r \in \R \bbar r < a\}, \quad B:=\{r \in \R \bbar r < b\}.
    \]
    And set $\alpha:= \sup A$ and $\beta:=\sup B$. By the completeness of $\R$ we know that $\alpha, \beta \in \R$. Let $z:= \alpha + \ii \cdot\beta$. We first show that $z^* - z \in \Lim_\C$: take any $\epsilon > 0$. Since $\alpha$ is an upper bound of $A$ we know $\alpha + \epsilon \notin A \Rightarrow a \leq \alpha + \epsilon$. Furthermore, $\alpha - \epsilon < a$, hence otherwise we have $a \leq \alpha - \epsilon$ which would be a lower upper bound of $A$ which is a contradiction to the construction of $\alpha$. So we have 
    \[
        \alpha - \epsilon < a \leq \alpha + \epsilon \Leftrightarrow |a-\alpha| \leq \epsilon.
    \]
We can argue the same way for $b$ and have $|b-\beta| \leq \epsilon'$ for some $\epsilon' >0$.
Finally, we can conclude 
\[
    |z^*-z| = \sqrt{(a-\alpha)^2 + (b-\beta)^2} \leq |a-\alpha| + |b-\beta| \leq \epsilon + \epsilon' =: \hat \epsilon,
\]
where we used that $\norm{x}_1 \geq \norm{x}_2 \forall x \in \R$ and by universal transfer also for all $x \in \Rs$.
Since this holds for all $\hat \epsilon > 0$ we know that $z^*$ and $z$ are infinitesimal close.

We still have to show the uniqueness of $z$. Assume there is another $z' \in \C$ with the same property. Then $z^* \simeq z'$ and therefore $z \simeq z'$. Since both $z$ and $z'$ are complex numbers this means $z=z'$.
\end{proof}

\begin{lemma}[Complex Shadow, \cite{nsa-dyn-geo}]
    \label{nsa_basics:complex_shadow}
   For $z \in \Lim_\C$ and $z = a + \ii \cdot b$ we have 
   \[
       \sh(z) = \sh(a) + \ii \cdot \sh(b)
   \]
\end{lemma}
\begin{proof}
    Direct consequence of the construction in the proof of \autoref{nsa_basics:shadow}.
\end{proof}
\begin{lemma}[Complex Shadow and Conjugation, \cite{nsa-dyn-geo}]
    \label{nsa_basics:conj_shadow}
    Let $z \in \Lim$, then it holds true that $\sh(\conj{z}) = \conj{\sh(z)}$.
\end{lemma}
\begin{proof}
    Apply \autoref{nsa_basics:complex_shadow}. 
\end{proof}
\begin{definition}[Almost Real, \cite{nsa-dyn-geo}]
    We call a number $z \in \Cs$ \textit{almost real}, if $z$ is infinitely close to a real number. That means there exists a $r \in \R : z \simeq r$. 
\end{definition}
\begin{lemma}[Shadow Properties, \cite{nsa-dyn-geo}]
    \label{nsa_basics:sh_prop}
Let $a,b \in \mathbb L$ and $n \in \mathbb \N$ then
\begin{enumerate}
        \item $\sh(a \pm b) = \sh(a) \pm \sh(b)$
        \item $\sh(a \cdot b) = \sh(a) \cdot \sh(b)$
        \item $\sh(\frac{a}{b}) = \frac{\sh(a)}{\sh(b)}, \text{ if }\sh(b) \neq
            0$
        \item $\sh(b^n )= \sh(b)^n$
        \item $\sh(|b|)= |\sh(b)|$
        \item for $a,b \in \Lim_\R:$ $\text{if } a \leq b \text{ then } \sh(a) \leq \sh(b)$
    \end{enumerate}
\end{lemma}
\begin{proof}
    For the real case see \cite{goldblatt} \page 53-54. For the complex case we can reduce this to the real case by \autoref{nsa_basics:complex_shadow}.
\end{proof}

\begin{theorem}[{Isomorphism{, partially \cite{goldblatt} \page 54}}]
    \label{nsa_basics:iso}
    The quotient ring $\mathbb{L_\R}/\mathbb{I_\R} \; (\mathbb{L_\C}/\mathbb{I_\C})$ is isomorphic to the field of the
    real (complex) numbers by $\hal(b) \mapsto \sh(b)$. Therefore $\mathbb I$ is a
    maximal ideal of the ring $\mathbb L$.
\end{theorem}
\begin{proof}
 For the real case \cite{goldblatt} \page 54. For the complex case, note that $\Cs \simeq \Rs \times \Rs$ and apply the same arguments as for the real case.
\end{proof}
\begin{theorem}[{Continuity{, partially \cite{goldblatt} \page 75}}]
    \label{nsa_basics:conti}
The function $f: \C \rightarrow \C$ is continuous at $c\in \C$, if and only if $f(c) \simeq f(x)$ for all $x
    \in \Cs$ such that $x \simeq c$. In other words if and only if 
    \[
        f(\hal(c)) \subset \hal(f(c)).
    \]
\end{theorem}
\begin{proof}
    The real case is again in \cite{goldblatt} \page 75. The complex case reads analogously since the definition of continuity is essentially the same only with a different definition of the absolute value.
\end{proof}
\begin{remark}
    This property will be crucial later on to resolve singularities in geometric constructions.
\end{remark}
\begin{lemma}[{Real Limits{, \cite{goldblatt} \page 78}}]
        \label{nsa_basics:real_limits}
    For $c, L \in \R$ and $f$ be defined on $A\subset \R$ then it holds true:
    \begin{align*}
        \lim_{x \rightarrow c} f(x) &= L \Leftrightarrow f(x) \simeq L \quad \forall x \in \As: x \simeq c,\; x \neq c\\
        \lim_{x \rightarrow c^+} f(x) &= L \Leftrightarrow f(x) \simeq L \quad \forall x \in \As: x \simeq c,\; x > c\\
        \lim_{x \rightarrow c^-} f(x) &= L \Leftrightarrow f(x) \simeq L \quad \forall x \in \As: x \simeq c,\; x < c\\
    \end{align*}
\end{lemma}
\begin{lemma}[Squeezing Limits, \cite{nsa-dyn-geo}]
    \label{nsa_basics:sqeezing_limits}
    Let $A\subset \R, c \in \inter(A)$, $L \in \C$ and let $f: A \setminus \{c \} \rightarrow \C$ be continuous. 
    If there exists a $\Dx \in \I_\R \Smz, \Dx > 0$ such that $f(c+\Dx) \simeq f(c-\Dx) \simeq L$, then the function $f$ can be continuously extended on $A$ with $f(c) = L$.
\end{lemma}
\begin{proof}
    Define the sets $A^+ = \{a \in A \bbar a > c\}$, $A^- = \{a \in A \bbar a < c\}$ and $H^\circ:= \hal(c) \setminus \{c\}$. Since $f$ is continuous it holds true by \autoref{nsa_basics:real_limits} 
    for all $h^+ \in f(H^\circ \cap (A^+)^*)$ that $h^+ \simeq f(c+\Dx)$ and furthermore for $h^- \in f(H^\circ \cap (A^-)^*)$ that $h^- \simeq f(c-\Dx)$. As by assumption $f(c+\Dx) \simeq f(c-\Dx)$ it holds true that $h \simeq f(c\pm \Dx)$ for all $h \in f(H^\circ)$. 
    
    Now we are almost done: it holds true that $(A\setminus \{c\})^* = \As \setminus \{c\}$ (see \cite{goldblatt}, \page~29). So if we want to extend $f$ to whole $\As$ we only have to define a value for $f(c)$: we define $f(c):= L$.
    Since $f(c\pm \Dx) \simeq L = f(c)$ it holds true that $f(x) \simeq L$ for all $x \in \hal(c)$. That is equivalent with the continuity of $f$ at $c$ by \autoref{nsa_basics:conti}. 
\end{proof}
\begin{remark}
    The last lemma is very useful since it gives us a recipe to continuously extend a function! If we have a discontinuity of $f$ at some point $c$, it is removable if and only if the shadows of $f(c+\Dx)$ and $f(c-\Dx)$ coincide and we can continue the function with the calculated shadow. 
\end{remark}
\begin{lemma}[{Complex Limits{, partially \cite{goldblatt} \page 78}}]
        \label{nsa_basics:limits}
    Let $c, L \in \C$ and $f$ be defined on $A\subset \C$. Then 
    \[
        \lim_{x \rightarrow c} f(x) = L \Leftrightarrow f(x) \simeq L \quad \forall x \in \As: x \simeq c,\; x \neq c.
    \]
\end{lemma}
\begin{proof}
    Essentially, the proof is based on the equivalence of the Weierstrass $\epsilon-\delta$ continuity and the limit definition of continuity. The theorem is stated for the real case in \cite{goldblatt} \page 78. The proof by transfer can be found on \page 75--76. The complex version reads completely analogously.
\end{proof}
The next theorem will give important facts about the topology. It turns out that a set is open if and only if it includes the halo of all its points.

Remember that a set $A$ is open if and only if $\inter A = A$ (see \cite{munkres2000topology} \page 95).
\begin{lemma}[Halo and a Metric, \cite{davis1977applied} \page 89]
  \label{halo_and_matric}
 If $d$ is a standard metric on $X$ then it holds true 
    \[
        \hal(x):= \{ y \in X \bbar d(x,y) \simeq 0 \}.
    \]
   And so we can write $q \simeq p$ for $p,q \in X^*$ if and only if $d(p,q) \simeq 0$.
\end{lemma}
\begin{theorem}[Topological Non-standard Continuity, \cite{davis1977applied} \page 79]
    \label{nsa_basics:nsa_topo_conti}
   Let $f$ map $X$ to $Y$ where $X,Y$ are both topological spaces. Let $p \in X$. Then $f$ is continuous at $p$ if and only if 
   \[
       q \simeq p \Rightarrow f(q) \simeq f(p).
   \]
\end{theorem}
\begin{remark}
    The characterization of continuity in the previous theorem reads exactly analogously to the definition of continuity in \autoref{nsa_basics:conti} and this is completely canonical since the previous theorem generalizes infinitesimal proximity to arbitrary topological spaces.
This can be used later on to define continuity directly in the projective space.
\end{remark}
\begin{lemma}[Constant on Open Sets, \cite{nsa-dyn-geo}]
    \label{nsa_basics:const_fct}
    Let $A\subset \C$ open, $c\in A$ and $f: A \rightarrow \C$ be a function on $A$. 
    \begin{enumerate}
\item If $f$ is constant on $\hal(c)$ then there is an open set $A' \subset A$ with $c \in A'$ such that $f$ is constant on $A'$.
\item If $f$ is constant on on $A$ then $f$ is also constant on $\As$.
    \end{enumerate}
\end{lemma}
\begin{proof}
We proof 1: if $f$ is constant on $\hal(c)$ the following statement is true:
   \[
       \exists \delta \in (\R^+)^*: \forall x \in \Cs:\left(|x-c| < \delta \Rightarrow f(x) = f(c)\right).
   \]
   Now we apply existential transfer and find the next statement also true:
   \[
       \exists \delta \in \R^+: \forall x \in \C:\left(|x-c| < \delta \Rightarrow f(x) = f(c))\right.
   \]
   Which says that there is an open set $A':=  \{x \in A \bbar |x-c| < \delta \}$ such that $f$ is constant on $A'$.

   Now we proof 2: if $f$ is constant on $A$, then the following sentence is true: there is $d \in \C$ such that
   \[ \forall x \in A: f(x) = d 
   \]
   Then apply universal transfer:
   \[ \forall x \in \As: f(x) = d
   \]
which is the claim.
\end{proof}
\begin{remark}
    As the converse argument this also implies that if a function $f$ does not vanish on the $\hal(c)$ then there is an $\epsilon >0, \epsilon \in \R$ such that the function $f$ also does not vanish on $B_\epsilon(c)$.
\end{remark}

\section{Non-standard Projective Geometry}
\label{chap:nsa-pg}
We will now introduce the usage of hyperreal and hypercomplex numbers in projective geometry. Surprisingly there is some preliminary work by A.\ Leitner \cite{leitner2015limits}. Leitner introduces projective space over the hyperreal numbers (not considering complex space) and uses this to proof properties of the diagonal Cartan subgroup of $\operatorname{SL}_n(\R)$.

  We also published an article on non-standard methods dynamic projective geometry: \cite{nsa-dyn-geo}. There some of the following concepts were already introduced.
\begin{definition}[$\KsP^d$]
Let $d \in \N$. We define $\KsP^d$ analogously to $\KP^d$:
    \[
        \KsP^d \; := \;  \dfrac{\Ks^d \setminus \{0\}}{\Ks \setminus \{0\}}
    \]
    where $\Ks$ denote the hyper (real or complex) numbers.
\end{definition}
\begin{definition}[Standard and Non-standard Projective Geometry]
    If we talk about projective geometry in $\KP^d$, we will refer to it as \textit{standard projective geometry} and for results in the extended field $\KsP^d$ we will refer to as \textit{non-standard projective geometry}.
\end{definition}
\begin{definition}[Similary, \cite{nsa-dyn-geo}]
    If two vectors $x,y \in \KP^2$ represent the same equivalence class, we will write $x \sim_{\KP^2} y$ and analogously for the enlarged space $\KsP^2$ we will write $x \sim_{\KsP^2} y$.
    \end{definition}
    \begin{remark}
       If the situation allows, we will sometimes also drop the ``$\sim$'' relation symbol and just write ``$=$''. 
    \end{remark}
    \begin{definition}[Representatives, \cite{nsa-dyn-geo}]
            \label{nsa_pg:representative}
            We call a representative $x$ of $[x] \in \KsP^d$ 
        \begin{itemize}
        \item    \textit{limited}, if and only if for all components $x_i$ of $x$ it holds true that $x_i \in \mathbb L$. 
        \item \textit{infinitesimal}, if and only if for all
    components $x_i$ of $x$ it holds true that $x_i \in \mathbb I$. 
\item    \textit{appreciable}, if and only if for all components $x_i$ of $x$ it holds true that $x_i \in \mathbb L$ and there is at least one component $x_{i'}$ which is appreciable. 
    \item \textit{unlimited}, if and only if for at least one
        $x_i$ of $x$ it holds true that $x_i \in \mathbb \K_\infty$. 
    \end{itemize}
\end{definition}
\begin{remark}
        We will denote by $\norm{x}$ the Euclidean norm $\norm x_2$ of a vector in $x \in \K^d$ or $x \in \Ks^d$, if not stated otherwise. 
\end{remark}
\begin{lemma}[Representatives and Norms]
	\label{nsa_pg:rep_norms}
	Let $x \in \CsP^d$, then it holds true:
    \begin{itemize}
	    \item  $x$ is limited, if and only if $\norm{x} \in \mathbb L$. 
	    \item  $x$ is infinitesimal, if and only if $\norm{x} \in\mathbb  I$. 
	    \item  $x$ is appreciable, if and only if $\norm{x} \in\mathbb  A$. 
	    \item  $x$ is unlimited, if and only if $\norm{x} \in\mathbb  \K_\infty$. 
    \end{itemize}
\end{lemma}
\begin{proof}
    Obvious.
\end{proof}
\begin{definition}[Points and Lines in $\KsP^2$]
    We define the sets of point $\mathcal{P}_\Ks$ and lines $\mathcal{L}_\Ks$ of $\KsP^2$ by
    \begin{align*}
        \mathcal{P}_\Ks \; := \;  \dfrac{\Ks^3 \setminus \{0\}}{\Ks \setminus
        \{0\}}, \quad 
        \mathcal{L}_\Ks \; := \;  \dfrac{\Ks^3 \setminus \{0\}}{\Ks \setminus \{0\}}.
    \end{align*}
\end{definition}
This definition is is the logical consequence of the introduction of hyperreal and hypercomplex numbers. But this also entails modifications to the standard operations in projective geometry. For example the cross product of two points defines the connecting line, but the definition will not be independent of the representative any more (as we will see later in \autoref{expl:geokalkuele}). This is not as bad as it seems on the first sight: it turns out that if representatives are of reasonable length (meaning appreciable), a lot of properties transfer from standard geometry to the non-standard version. Therefore, we will introduce adapted definitions for the basic operations in projective geometry.

Firstly, we want to measure whether two numbers are ``approximately'' of the same size. We already know the concept of galaxies but this is not that useful in our case. We rather would like to have something which measures if the quotient of two numbers yields an appreciable number, \ie whether the shadow of the quotient is in $\C$ and not $0$. Hence, we define the magnitude of complex numbers:
\begin{definition}[Magnitude]
    \label{nsa_basics:def_magni}
   We define the \textit{mangitude} of a number $r \in \Cs \setminus \{0\}$ by
    \[
    \magni (r) := \{ s \in \CsZ \; | \; \exists A \in \mathbb A: r = A \cdot s 
    \}.
    \]
\end{definition}
\begin{remark}
    A.\ Leitner has a similar concept and calls two numbers of same order if their quotient is appreciable \cite{leitner2015limits}. 
\end{remark}
\begin{example}[Magnitude Examples]
	Let $x \in \C \setminus \{0\}, H \in \C_\infty, \epsilon \in \I\setminus \{0\}$
	\begin{align*}
		&5\cdot H \notin \gal(H) \quad (5\cdot H - H = 4 \cdot H \notin \Lim)\\
        &5\cdot H \in \magni (H) \quad (\frac{5\cdot H}{H} = 5 \in \A)\\
        &\epsilon^2 \notin \magni( \epsilon) \quad
        (\frac{\epsilon^2}{\epsilon} = \epsilon \notin \A) \\
		& \Lim \setminus \{0\} = \magni (x) 
	\end{align*}
\end{example}
\begin{lemma}[Appreciable and Magnitude]
    \label{nsa_pg:apr_magni}
    Let  $r \in \Cs \setminus \{0\}$ then 
    \[
    s \in \magni(r) \Leftrightarrow
    \sh(\frac{r}{s}) \in \C \setminus \{0\}\Leftrightarrow
    \sh(\frac{s}{r}) \in \C \setminus \{0\}.
    \]
\end{lemma}
\begin{proof}
Let $s \in \magni (r)$ then there is an $A \in \A$ such that $r = A \cdot s$,
so $A = \frac{r}{s}$. By definition there is a $n \in \N$ such that
$\frac{1}{n} < |A| < n$ \ie there is a $x \in \C \setminus \{0\}$ such that
$\sh(A) = x$. Furthermore the inverse $A^{-1} \in \A$ since $\frac{1}{n} <
    |\frac{1}{A}| < n$ just by definition of the reciprocal.
\end{proof}
\begin{lemma}[Magnitude Lemmas]
    \label{sing:magn_lemma}
    Let $r, \lambda \in \Cs \setminus \{0\}$ and $\sigma \in \mathbb A$, then
\begin{align}
 \magni (r) &= \magni(\sigma \cdot r) \label{sing:magn_lemma:1},\\
 \magni (r)^{-1}&:= \{ \frac 1 s \bbar s \in \magni(r) \}  = \magni(r^{-1}) \label{sing:magn_lemma:2},\\
 \magni (\lambda \cdot r) &= \lambda \cdot \magni(r)   \label{sing:magn_lemma:3}.
\end{align}    
\end{lemma}
\begin{proof}
    (\ref{sing:magn_lemma:1}): $s \in \magni (r) \Leftrightarrow \exists A \in
    \A: r = A \cdot s \Leftrightarrow  \sigma \cdot r = \sigma \cdot A
    \cdot s $. Since $A, \sigma \in \A$ their product $A':= \sigma \cdot A \in \A$ 
    too, which means $\sigma \cdot r = A' \cdot s$. Therefore $s \in \magni (\sigma \cdot
    r) \Leftrightarrow s \in \magni (r)$.\\
    (\ref{sing:magn_lemma:2}): Let $s \in \magni (r^{-1}) \Leftrightarrow \exists  
    A \in \A: r^{-1} = A \cdot s \Leftrightarrow r = A^{-1}s^{-1} \Leftrightarrow s^{-1} \in \magni (r)$.\\
    (\ref{sing:magn_lemma:3}): \begin{align*}
        \lambda \cdot \magni (r) &= \lambda \cdot \{s \in \Cs \setminus\{0\} \; | \; \exists A \in \A: r = s \cdot A \} \\ 
        &= \{  \underbrace{\lambda \cdot s}_{=:s'} \in \Cs \setminus\{0\} \; | \; \exists A \in \A: r = s \cdot A \} \\ 
&= \{  s' \in \Cs \setminus\{0\} \; | \; \exists A \in \A: r = s' \cdot
    \lambda^{-1} \cdot A \} \\ 
&= \{  s' \in \Cs \setminus\{0\} \; | \; \exists A \in \A: \lambda \cdot r = s' \cdot A \} \\ 
    &= \magni (\lambda \cdot r)
    \end{align*}
\end{proof}
\begin{lemma}[Magnitude and Limited Numbers]
    \label{nsa_pg:mag_lim}
    Let $r, \lambda \in \Cs \setminus \{0\}$ and $s \in \magni (r)$ then
    \[
        s \in \magni (\lambda \cdot r) \Leftrightarrow \lambda \in \A
    \]
\end{lemma}
\begin{proof}
    ``$\Leftarrow$'' we already showed in \autoref{sing:magn_lemma}.\\
    ``$\Rightarrow$'' 
    Let $s
    \in \magni (\lambda \cdot r) \Leftrightarrow \exists A \in \A: \lambda \cdot
    r = A \cdot s \Leftrightarrow \frac{r}{s} = \frac{A}{\lambda}$. If now
    $\lambda \in \I$ the quotient $\frac{A}{\lambda} \in \C_\infty$ and if
    $\lambda \in \C_\infty$ we have $\frac{A}{\lambda} \in \I$ (both by
    \autoref{nsa_basics:arith}). This is a contradiction to the assumption.
    
    The only case left
    is the appreciable one: since the reciprocal of a appreciable number
    is appreciable, also \autoref{nsa_basics:arith}, the claim follows
    immediately.
\end{proof}
\begin{theorem}[Limited and Appreciable Entries]
    \label{nsa_pg:lim_entry}
    Let $x = (x_1,\ldots,x_d)^T \in \CsP^d$ and let $\lambda \in \magni(\norm{x})$ then
    all entries of $\lambda^{-1} x$ are limited and at least one is appreciable. 
\end{theorem}
\begin{proof}
   Let $\hat x$ be the absolute maximum of $x$. Then we know that by \autoref{nsa_basics:transfer} that the Cauchy-Schwarz inequality holds true: \[|\hat x| =
   \sqrt{\hat x\cdot\conj{\hat x}} \leq \sqrt{x_1\cdot\conj{x_1} + \ldots + x_d\cdot\conj{x_d}} = \norm{x}.\] Since $\lambda \in
   \magni(\norm{x})$ there is an $A \in \mathbb A$ such that $\norm{x} = |A| \cdot
   |\lambda|.$ So $|\hat x \lambda^{-1}| \leq \norm{x} |\lambda|^{-1} = |A|$ which
   means the $\hat x \lambda^{-1}$ has to be limited. Since $\hat x$ was the absolute maximum all other entries have to be limited, too.

   Now we proof that at least one entry has to be appreciable. By the arguments above we know that $\hat x \lambda^{-1}$ is limited, therefore we
   have to show that $\hat x \lambda^{-1}$ can't be infinitesimal.  Assume the
   converse. Then all entries of the vector $x \lambda^{-1}$ have to be
   infinitesimal by \autoref{nsa_pg:rep_norms}. And so the norm $\norm{x \lambda^{-1}}$ is also infinitesimal
   equal to an $\epsilon \in \mathbb I$.
   Then consider 
   \begin{align*}
       1 &= \norm{\frac{x \lambda^{-1}}{\norm{x \lambda^{-1}}}} = \norm{\frac{x
       \lambda^{-1}}{\norm{x } |\lambda^{-1}|}} = \norm{\frac{x
       \lambda^{-1}}{|A||\lambda||\lambda^{-1}|}} = \norm{\frac{x
       \lambda^{-1}}{|A|}}  \\
   \end{align*}
   Therefore it holds true $|A| = \norm{x \lambda^{-1}} = \abs{\epsilon}$ which is a
 contradiction to $A \in \mathbb A$.
\end{proof}
The last theorem motivates the definition of an appreciable representative:
\begin{definition}[Appreciable Representative, \cite{nsa-dyn-geo}]
    For an element $x = (x_1,\ldots,x_d)^T \in \CsP^d$ a \textit{appreciable
    representative} is defined as
    \[
        x_{\mathbb A} := \frac{1}{\lambda} \begin{pmatrix}
           x_1 \\ \vdots \\ x_d 
   \end{pmatrix} 
    \]
    with $\lambda \in \magni (\norm{x})$. Then all entries of $x_{\mathbb A}$
    are limited and at least one entry is appreciable.
\end{definition}
\begin{remark}
        While the notion of the magnitude might look a bit cumbersome, at first it is actually a fairly easy concept. The magnitude of the norm of a vector is just a way to ensure that we rescale the vector in a way that it is appreciable, meaning that there is a shadow in $\C^{d+1} \Smz$ which yield a viable representative of a equivalence class in $\CP^d$. 
        
        It would be possible to divide by the Euclidean norm of the vector and have a vector of length one, but the magnitude gives you more flexibility. Still it is sometimes preferable to divide, for example, by the absolute maximum of a vector, which is in the magnitude of the Euclidean norm.  
\end{remark}
\begin{lemma}[Equivalence of Euclidean and Maximum Norm]
    \label{nsa_pg:equiv_norm}
    For $z \in \Cs^d$ the Euclidean and the maximum norm are standard equivalent. This means that there are $a, A, b, B \in \R^+ \Smz$ such that
    \[
        a \norm{z}_\infty \leq \norm{z}_2 \leq  A \norm{z}_\infty
    \]
    and 
    \[
        b \norm{z}_2 \leq \norm{z}_\infty \leq  B \norm{z}_2
    \]
\end{lemma}
\begin{proof}
    Analogously to the direct proofs in $\C^d$, see for example \cite{horn2013matrix}.
\end{proof}
\begin{remark}
   The important thing here is that $a, A, b, B \in \R^+ \Smz$
are real and not hyperreal.
\end{remark}
\begin{lemma}[Maximum Norm and Magnitude]
    For a vector $z \in \Cs^d \Smz$ it holds true that $\norm{z}_\infty \in \magni(\norm{z})$.  
\end{lemma}
\begin{proof}
    By \autoref{nsa_pg:equiv_norm} there are $a,A \in \R^+ \Smz $ such that 
    \begin{align*}
        &a \norm{z}_\infty \leq \norm{z}_2 \leq  A \norm{z}_\infty\\
        \Leftrightarrow \; &a \leq \frac{\norm{z}_2}{\norm{z}_\infty} \leq  A 
    \end{align*}
    which means that with $c := \frac{\norm{z}_2}{\norm{z}_\infty}$ it holds true 
    \[
        c = \frac{\norm{z}_2}{\norm{z}_\infty}\\
 \Leftrightarrow \; c \cdot {\norm{z}_\infty} = {\norm{z}_2}
 \Rightarrow  \norm{z}_\infty \in \magni( {\norm{z}_2})
    \]
\end{proof}
\begin{remark}
    So instead of normalizing with the Euclidean norm we can also use the absolute maximum value for the projective shadow. Even more conveniently, we can use the $\argmax z_i$ of a hypercomplex vector $z \in \Cs^d \Smz$ since the absolute value function and the $\argmax$ differ only by an appreciable phase value, \ie $\argmax z_i = \e^{\phi} \cdot |z_i|$ for an $\phi \in [0,2 \pi)$.
\end{remark}
\begin{lemma}[Different Appreciable Representatives, \cite{nsa-dyn-geo}]
    \label{nsa_pg:diff_repr}
  Let $x \in \CsP^d$, $x_\A$ and $\hat x_\A$ be appreciable representatives. Then
  there is a unique $c \in \A$ such that $\hat x_\A = c \cdot x_\A$.
\end{lemma}
\begin{proof}
    By the definition of equivalence classes in $\CsP^d$, we know that there is
    a unique $c \in \Cs \setminus \{0\}$ such that $\hat x_\A = c \cdot
    x_\A$. We have to prove that $c \in \A$: assume the converse. Then
    $c \in \I$ or $c \in \C_\infty$. We know by 
    \autoref{nsa_pg:lim_entry} that all entries are limited and at least one is
    appreciable. Since all entries are limited we know that $c$ can't be
    unlimited, otherwise the product of a limited entry with $c$ would be
    unlimited. But $c$ cannot be infinitesimal either, since the product of
    limited entries with an infinitesimal is infinitesimal by 
    \autoref{nsa_basics:arith} which is a contradiction to appreciability of at
    least one component.
\end{proof}
\begin{lemma}[Magnitude and Appreciability]
        \label{nsa_pg:mag_appr}
        Let $x \in \CsP^d$ and let $\lambda \in \Cs \setminus \{0 \}$. If all components of ${\lambda}^{-1}x$ are limited and at least one component is appreciable, then $\lambda \in \magni{(\norm x)}$. 
\end{lemma}
\begin{proof}
    We know that all entries of $x_\A:=\frac{x}{\norm{x}}$ are limited and at least one component is appreciable. Furthermore we know there is a $c \in \Cs \Smz$ such that:
        \[
            \lambda^{-1} x  = c \cdot \frac{x}{\norm{x}}        
    \]
and additionally $c \in \A$, since otherwise $\lambda^{-1} x$ would have an appreciable component. It follows that
\begin{align*}
    \lambda &= \frac 1 c \norm{x} \\ \Leftrightarrow \frac {\lambda}{ \norm{x}} &= \frac 1 c.
\end{align*}
Since  $\frac 1 c \in \A$ this means that $\lambda \in \magni(\norm{x})$, which is the claim.
\end{proof}
\begin{remark}
    This argument can be easily generalized to matrix norms which will be useful later on when we discuss conic sections and projective transformations (see \autoref{nsa_pg:appr_matrix}).
\end{remark}
\begin{definition}[Projective Shadow, \cite{nsa-dyn-geo}]
    \label{nsa_pg:psh}
    The \textit{projective shadow} of $[x] = [(x_1,\ldots,x_d)^T] \in \CsP^d$ is defined by
    \[
        \psh ([x]) := [\sh (x_{\mathbb A})]    \]
        where $x_\A$ is an appreciable representative of $x$.
\end{definition}
\begin{remark}
    Writing out the definition above:
    \[
\psh([x]) =  \left[ \sh \left( \frac{1}{\lambda} \begin{pmatrix}
           x_1 \\ \vdots \\ x_d 
   \end{pmatrix} \right) \right]= \left[ \begin{pmatrix}
      \sh \frac{x_1}{\lambda} \\ \vdots  \\
  \sh\frac{x_d}{\lambda}  \end{pmatrix} \right]
  \]
    with $\lambda \in \magni (\norm{x})$ and $\sh$ the standard shadow of $\Cs$.
\end{remark}
\begin{remark}
    A slightly less general definition was already given by A.\ Leitner for the real non-standard projective space \cite{leitner2015limits}.
\end{remark}
\begin{remark}
    One might be tempted to use \autoref{nsa_basics:sh_prop}, especially something like: $\sh\left(\frac{x_i}{\lambda}\right) = \frac{\sh(x_i)}{\sh(\lambda)}$ for the projective shadow. This it is wrong in a general setting! First of all, \autoref{nsa_basics:sh_prop} is only valid for limited numbers and especially the statement of fractions only holds true, if the shadow of $\lambda$ is not zero, \ie $\lambda$ is not infinitesimal.
\end{remark}
\begin{theorem}[$\CP^d$ Isomorphism]
    The space $\CsP^d/\mathbb I$ is isomorphic to $\CP^d$ with the
    correspondence map $\psh$.
\end{theorem}
\begin{proof}
    Let $x \in \CsP^d$. We know by \autoref{nsa_pg:lim_entry} that all entries of $x_\A$ are
    limited. Therefore we already know that for every component we can apply
    \autoref{nsa_basics:iso} and use that $\mathbb L / \mathbb I$ is
    isomorphic to $\C$. Furthermore we know that at least one component is
    appreciable and not
    infinitesimal again by \autoref{nsa_pg:lim_entry}. This ensures that at
    least one component not in the halo of 0. Therefore we know the shadow of an appreciable representative $x_\A$ is element of $\CP^d$.
\end{proof}
Incidences play a very important role in projective geometry. For points and lines these can  
conveniently be determined by the standard scalar product in $\CP^2$. We will define a
generalization of the scalar product here in order to embrace infinitesimal and unlimited
representatives.
\begin{definition}[$\CsP^2$ Appreciable Scalar Product]
    \label{nsa_pg:scal_prod}
    We define the \textit{appreciable scalar product} $\langle \cdot, \cdot \rangle_*$ for objects $x,y \in \CsP^2$ as
    \[
    \langle x, y \rangle_*:= \langle x_\A, y_\A\rangle,
    \]
    with arbitrary appreciable representatives $x_\A, y_\A$, where $\langle \cdot, \cdot \rangle$ denotes the standard complex scalar product.
\end{definition}
\begin{remark}
       For the case of $\RsP^2$ one would of course use the real scalar product. 
\end{remark}
    The value of the scalar product is not independent of the
    representatives, nevertheless if the value of the ``membership'' of the scalar
    product is infinitesimal is well defined \ie $\langle x,y \rangle_* \in
    \I$ is independent of the choice of an appreciable representative. This is
    our next task:

\begin{theorem}[Well-defined $\I$ Relation]
    \label{nsa_pg:i_well_def}
    For $x,y \in \CsP^2$ the relation \[\langle x, y \rangle_* \in \I\] is well defined.
\end{theorem}
\begin{proof}
Let $x_\A, \hat x_\A$ be limited representatives of $x$ and $y_\A, \hat y_\A$ be limited representatives of $y$. By \autoref{nsa_pg:diff_repr} we can choose $\iota, \kappa \in \A$  such that $x = \iota \cdot \hat x$, $y = \kappa \cdot \hat y$.
Then it holds true:
\[
    \langle x, y \rangle_* = \langle x_\A, y_\A \rangle = \langle \iota \cdot \hat x_\A, \kappa \cdot \hat y_\A \rangle = \underbrace{\bar \iota \cdot \kappa}_{=:c} \; \langle  \hat x_\A, \hat y_\A \rangle 
\]
Since $\kappa, \iota \in \A$ it holds true that also $c \in \A$. And so 
\[
     \langle x_\A, y_\A \rangle  \in \I \Leftrightarrow \langle  \hat x_\A, \hat y_\A \rangle \in \A,
\]
which is the claim.
\end{proof}
\begin{theorem}[Appreciable Scalar Product and Shadows]
    \label{nsa_pg:scalar_sh}
For $x,y \in \CsP^2$ it holds true
\[
    \sh\left(\langle x, y \rangle_*\right) = 0 \Leftrightarrow \langle \psh(x),
    \psh(y) \rangle = 0 
\]
\end{theorem}
\begin{proof}
    Let $x_\A, \hat x_\A$ be appreciable representatives of $x$ and $y_\A, \hat y_\A$ be appreciable representatives of $y$. Then by \autoref{nsa_pg:diff_repr} we know there are $\iota, \kappa \in \A$ such that $x_\A = \iota \cdot \hat x_\A$ and $y_\A = \kappa \cdot \hat y_\A$. Then we find
    \begin{align*}
    \langle \psh( x), \psh( y) \rangle &= \langle \sh(\hat x_\A), \sh(\hat y_\A) \rangle = \overline{\sh(x_\A)}^T \cdot \sh(y_\A) \\
                                       &= \sum_{i=1}^3 \sh(\overline{x_{\A_i}}) \cdot \sh(y_{\A_i}) = \sum_{i=1}^3 \sh\left(\overline{x_{\A_i}} \cdot y_{\A_i}\right) \\
                                       &= \sh \left( \sum_{i=1}^3 \overline{x_{\A_i}} \cdot y_{\A_i}\right) = \sh(\langle x_\A, y_\A \rangle ) \\
                                       &=  \underbrace{\sh(\bar \iota \cdot \kappa)}_{:=c} \; \sh(\langle \hat x_\A, \hat y_\A \rangle ) = c \; \langle \hat x_\A, \hat y_\A \rangle = c \; \langle x,y \rangle_*.
    \end{align*}
    Since $c \in \A $ it is easy to see that
\[
    \sh\left(\langle x, y \rangle_*\right) = 0 \Leftrightarrow \langle \psh(x)
    \psh(y) \rangle = 0.
\]
\end{proof}

\begin{lemma}
    Let $x,y \in \CsP^2$, if we choose the same appreciable representative in
    $\sh\left(\langle x, y \rangle_*\right)$ and $\langle \psh(x), \psh(y)
    \rangle$ respectively it even holds true:
\[
    \sh\left(\langle x, y \rangle_*\right)  = \langle \psh(x),
    \psh(y) \rangle  
\]
\end{lemma}
\begin{proof}
    We only have to realize that in the proof of \autoref{nsa_pg:scalar_sh}
    $c = 1$ if we choose the same representatives. Then the claim
    immediately follows.
\end{proof}
\begin{remark}
    Some of the following ``almost relations'' in the following sections were already investigated by J.~Fleuriot in an Euclidean setting (see for example \cite{fleuriot2012, fleuriot2000nonstandard}). We have to do a little more work in projective space than in the Euclidean setting since the representation of a vector plays an important role.
\end{remark}
\begin{definition}[Almost Orthogonal, \cite{nsa-dyn-geo}]
    We two vectors $l,p \in \Cs^3$ \textit{almost orthogonal} and write $[p] \perp_\I [l] $  if and only if
    \[
 \langle p, l \rangle_*       = \epsilon \in
        \mathbb{I}.
    \]
\end{definition}
\begin{remark}
    Note that if two vectors in $\C^3$ are orthogonal, they are also almost
    orthogonal since $0$ is the only infinitesimal in $\C$.
\end{remark}
\section{Incidence of Points and Lines}
Now we will start out with basic projective geometry: we define incidence relations of points and lines and analyze well known properties from the standard world.
\begin{definition}[Almost Incident]
    Define relation $\mathcal{I}_\Cs \subset \mathcal{P}_\Cs \times \mathcal{L}_\Cs$ for a point $p \in  \mathcal{P}_\Cs$ and a line $l \in  \mathcal{L}_\Cs$ by
    \[
        [p] \mathcal{I}_\Cs [l] \Leftrightarrow  [p] \perp_\I [l].
    \]
    Then we call $p$ \textit{almost incident} to $l$.
\end{definition}
\begin{remark}
    The incidence relation of $\CP^2$ is a special case of the almost incidence
    relation above \ie if $p$ and $l$ are incident in $\CP^2$ then they are
almost incident in $\CsP^2$.
\end{remark}
\begin{lemma}[Well-defined Incidence]
    \label{nsa_pg:well_def}
    The incidence relation is well defined.
\end{lemma}
\begin{proof}
Special case of \autoref{nsa_pg:i_well_def}.
\end{proof}
\begin{lemma}[Incident and Almost Incident]
    \label{nsa_pg:inci_and_almost}
    A point $p$ and a line $l$ in $\CsP^2$ are almost incident if and only of
    their projective shadow $\psh(p)$ and $\psh(l)$ are incident in $\CP^2$.
\end{lemma}
\begin{proof}
    Note that the shadow of an infinitesimal number is $0$. Then the claim it a direct application of \autoref{nsa_pg:scalar_sh}. 
\end{proof}
\begin{example}[Geometriekalküle]
    \label{expl:geokalkuele}
    In ``Geometriekalküle'' by J. Richter-Gebert and T. Orendt
    \cite{richter2009geometriekalkule} there is a nice example on the
    introduction of farpoints (see \page 4). Let $x,y \in \C$, and define a point $P(t)$ by: 
   \[
   [P(t)]= \left[\begin{pmatrix} x \cdot t \\ y \cdot t \\ 1\end{pmatrix}
   \right] = \left[\begin{pmatrix} x \\ y \\ 1/t\end{pmatrix} \right].
   \]
   If one takes the limit of $t$ to infinity, one finds:
   \[
   \lim_{t \rightarrow \infty} [P(t)]= \lim_{t \rightarrow \infty}
   \left[\begin{pmatrix} x \\ y \\ 1/t\end{pmatrix} \right] =
   \left[\begin{pmatrix} x \\ y \\ 0\end{pmatrix} \right].
   \]
   One can easily check that this point is incident with the line at infinity $\linf$ of the standard embedding. 

   We can give an analogous example using hyperreal numbers. Let H be an arbitrary real unlimited number and define $P'$:
   \[
   [P']:= \left[\begin{pmatrix} x \cdot H \\ y \cdot H \\ 1\end{pmatrix} \right].
   \]
   $P'$ is a point which should be incident to $l_\infty$. However, then the usual check of incidence, the standard scalar product, fails: 
   \[ 
       \langle  \begin{pmatrix} x \cdot H \\ y \cdot H \\ 1\end{pmatrix},
       \begin{pmatrix} 0 \\ 0 \\ 1\end{pmatrix}\rangle = 0 \cdot x \cdot H + 0 \cdot y \cdot H + 1 =
       1 \not \in \I.
   \]
   Now if we note that $H$ is in $\magni (\norm{P})$, then we find:
   \[ 
       \langle  \begin{pmatrix} x \cdot H \\ y \cdot H \\ 1\end{pmatrix},
       \begin{pmatrix} 0 \\ 0 \\ 1\end{pmatrix}\rangle_* = 
       \langle  \begin{pmatrix} x \\ y \\ \frac{1}{H}\end{pmatrix},
       \begin{pmatrix} 0 \\ 0 \\ 1\end{pmatrix}\rangle = \frac{1}{H}  
   \]

   Since H is unlimited, its inverse is infinitesimal and therefore the point is
   almost incident to the line at infinity. This means that the projection of
   $P'$ to $\CP^2$ is incident to $\linf$ and hence a farpoint as expected.
\end{example}
\begin{remark}
As for the appreciable scalar product in \autoref{nsa_pg:scal_prod}, we also have to say how we find the connecting line of two points or the intersection of two lines, respectively. The standard cross-product is fine, if we want to work in $\CsP^2$, but if we want to draw conclusions about $\CP^2$ (\ie taking the show) we have to redefine the operator and normalize beforehand.
\end{remark}
\begin{definition}[Appreciable Cross-Product, \cite{nsa-dyn-geo}]
    \label{nsa_pg:cross-prod}
    Let $x,y$ be in two vectors in $\CsP^2$. We define the \textit{appreciable cross-product} $\mathit{\times_*}:$ by:
   \[
       x \times_* y :=  x_\A \times y_\A 
   \]
   For appreciable representatives $x_\A, y_\A$ of $x$ and $y$.
\end{definition}
\begin{lemma}[Well-defined Cross-Product]
    \label{nsa_pg:cross_well_defined}
    The appreciable cross-product defined in \autoref{nsa_pg:cross-prod} is well defined. 
\end{lemma}
\begin{proof}
 By the bilinearity of the cross-product the claim is obvious.
\end{proof}
\begin{remark}
    We will see in \autoref{nsa_pg:appr_cross_expl} that we have to be careful with the application of shadow function if it is applied to the cross-product.
\end{remark}
\begin{definition}[Almost Far Point, \cite{nsa-dyn-geo}]
    We call $p \in \mathcal{P}_\Cs$ an \textit{almost far point}, if it holds true for an appreciable representative $p_\A$ of $p$ with 
    \[
        p_\A = \begin{pmatrix} x \\ y \\ z\end{pmatrix} \quad \text{ and } x, y, z  \in \Lim
    \]
    that $z \in \I$, \ie that the $z$-component of the appreciable representative is infinitesimal.
\end{definition}
\begin{definition}[Appreciable Join and Meet, \cite{nsa-dyn-geo}]
    Analogously to the standard definition of the join in $\CP^2$ we define the \textit{appreciable $\mathit{join}$} for two points $p,q \in \mathcal{P}_\Cs$ by
        \[
            \join_*(p,q):= p \times_* q = p_\A \times q_\A.
        \]
        And for two lines $l,m \in  \mathcal{L}_\Cs$ we define the \textit{appreciable $\mathit{meet}$} by 
        \[
            \meet_*(l,m):= l \times_* m = l_\A \times m_\A.
        \]
\end{definition}
\begin{remark}
    The result of an appreciable $\join$ or $\meet$ is not necessarily appreciable! Take for example the two points $p = (0,0,1)$ and $q = (\epsilon, 0,1)$, with $\epsilon \in \I$, \ie $\epsilon$ is infinitesimal. Then their appreciable $\join$, which is equivalent to the standard join since $p$ and $q$ are appreciable, is 
    \[
    \join_*(p, q) = \join(p,q) = \begin{pmatrix} 0 \\ \epsilon \\ 0 \end{pmatrix}
    \]
    which is not an appreciable vector!
\end{remark}
\begin{definition}[Almost Parallel]
    \label{nsa_pg:almost_par_def}
    Let $l,m \in \mathcal{L}_\Cs$ be two lines. 
    We call $l$ and $m$ \textit{almost parallel} if an appreciable representative of $\meet_*(l,m)$ is an almost far point.
\end{definition}
\begin{lemma}[Almost Parallel is well-defined]
    The definition of almost parallel in \autoref{nsa_pg:almost_par_def} is well-defined.
\end{lemma}
\begin{proof}
    Special case of \autoref{nsa_pg:i_well_def}.
\end{proof}
\begin{lemma}[Almost Far Points on $l_\infty$]
    \label{nsa_pg:almost_para_linf}
    A point $p$ is an almost far point if and only if $p$ is almost incident to $l_\infty = (0,0,1)^T$. 
\end{lemma}
\begin{proof}
    If $p$ is an almost far point then we can write an appreciable representative $p_\A$ with $p_\A = (x,y,\epsilon)$ where $x,y \in \Lim$ and $\epsilon \in \I$. Then the appreciable scalar product $\langle p_\A, l_\infty \rangle = 0 + 0 + \epsilon = \epsilon \in \I$. This means that $p$ is almost incident to $l_\infty$.

    If $p$ is almost incident to $l_\infty$ we know that for an appreciable representative $p_\A = (x,y,z)^T$ it holds true $\langle p_\A, \linf \rangle = z \in \I$. Therefore $p$ is an almost far point.
\end{proof}
\begin{lemma}[Almost Parallel Lemma]
    Two lines $l,m \in \mathcal{L}_\Cs$ are almost parallel if and only if it holds true:
    \[
        \meet_*(l,m)  \perp_\I l_\infty
    \]
    where $l_\infty = (0,0,1)^T$ denotes the line at infinity. In other words the intersection of both lines is almost incident to the line at infinity.
\end{lemma}
\begin{proof}
    Special case of \autoref{nsa_pg:almost_para_linf}.
\end{proof}

\begin{definition}[Shadow Cross-Product]
    \label{nsa_pg:sh-cross-prod}
    Let $x,y$ be two objects in $\KsP^2$. We define the \textit{shadow cross-product} $\times_{\sh}$ by:
   \[
       x \times_{\sh} y :=  \sh( x \times_* y )= \sh( x_\A \times y_\A ).
   \]
\end{definition}
\begin{lemma}[Well-defined Shadow Cross-Product]
    \label{nsa_pg:sh_cross_well_defined}
    The shadow cross-product defined in \autoref{nsa_pg:sh-cross-prod} is well-defined. 
\end{lemma}
\begin{proof}
    Let $x,y \in \KsP^2$, pick now two appreciable representatives $x_\A, \hat x_\A$ and $y_\A, \hat y_\A$. Then we know by \autoref{nsa_pg:diff_repr} that there are appreciable numbers $c,d$ such that
    \[
        x_\A = c \cdot \hat x_\A, \quad y_\A = d \cdot \hat y_\A.
    \]
    \[
        x \times_* y = x_\A \times y_\A = c \cdot \hat x_\A \times d \cdot \hat y_\A = (c \cdot d)  \cdot (\hat x_\A \times  \cdot \hat y_\A )     \]
    Where we used the bilinearity of the cross-product. Then, if we apply the shadow function and use the linearity of the shadow for limited numbers we know:
    \[
        \sh (x\times_* y) = \underbrace{\sh(c \cdot d)}_{\in \K \Smz} \cdot \sh(\hat x_\A \times   \hat y_\A)   \sim_{\KP^2} \sh(\hat x_\A \times   \hat y_\A) = \hat x \times_{\sh} \hat y
    \]
    where we used that $c \cdot d$ is appreciable.
\end{proof}
\begin{definition}[Almost Linearly Dependent]
    A set of objects in $\{v_1, \ldots, v_m \} \subset \Ks^n$ is called \textit{almost linearly dependent}, if there are $\lambda_1, \ldots, \lambda_m \in \Ks \setminus \I$ such that 
    \[
        \sum_{i=1}^m \lambda_i \cdot v_i \simeq 0.
    \]
    and we call them \textit{linearly dependent} if it even holds true that $\sum_{i=1}^m \lambda_i \cdot v_i = 0$. Otherwise, we call them \textit{linearly independent}.
\end{definition}
\begin{definition}[Almost Equivalent]
    We call two objects $x,y \in \CsP^n$ \textit{almost equivalent}, if
    $\exists \lambda \in \Cs$ such that
    \[
        x_\A \simeq \lambda \cdot y_\A
    \]
which means that the appreciable representatives of $x$ and $y$ are almost linearly dependent. We write again $x \simeq y$.
\end{definition}
\begin{definition}[Projective Halo]
    For a point $x \in \CsP^n$ we define the \textit{projective halo $\phal$ of $\mathit{x}$} by 
    \[
        \phal(x):= \{ y \in \CsP^n \bbar x \simeq y\}.
    \]
\end{definition}
\begin{lemma}[Appreciable Scaling Factor]
   Let $x,y \in \CsP^n$ be almost equivalent, \ie $\exists \lambda \in \Cs \setminus \I$ such that $x_\A \simeq \lambda \cdot y_\A$. Then it holds true that $\lambda$ is appreciable, \ie $\lambda \in \A$.
\end{lemma}
\begin{proof}
    \[
        x_\A \simeq \lambda \cdot y_\A \Rightarrow \norm{ x_\A } \simeq \abs{\lambda} \norm{y_\A} \\
        \Rightarrow \abs{\lambda} \simeq \frac{\norm{ x_\A }}{\norm{ y_\A }} \in \A. 
    \]
\end{proof}
\begin{lemma}[Almost Equivalent and Angles]
    \label{nsa_pg:aeq_angles}
Two elements $p, p' \in \CsP^2$ are almost equivalent, if and only if the angle $\alpha$ between $p$ and $p'$, interpreted as vectors in $\Cs^3$, is infinitesimal or infinitely close to $\pi$.
\end{lemma}
\begin{proof}
Pick appreciable representatives $p_\A$ and $p'_\A$ of $p$ and $p'$. Then by definition there are $c, c' \in \A$ such that 
    \[
        p_\A = \frac{p}{c \norm{p}} \quad \text{and} \quad p'_\A = \frac{p}{c' \norm{p'}}.
    \]
   Since $p,p'$ are almost equivalent there is a $\lambda \in \A$ such that $p' \simeq \lambda p$. Then it holds true 
    \[
        p_\A \times p'_\A =   \frac{p}{c \norm{p}} \times  \frac{p'}{c' \norm{p'}} = \underbrace{\frac{1}{c \cdot c'}}_{=:d} \left( \frac{p}{\norm{p}} \times  \frac{p'}{\norm{p'}} \right). 
    \]
   Note that $d \in \A$ since the product of appreciable numbers is appreciable and also its inverse.

By the properties of the scalar product and transfer we know:
\[
   \norm{p_\A \times p'_\A} = \abs{d} \norm{ \frac{p}{\norm{p}}} \norm{ \frac{p'}{\norm{p'}} } \sin(\alpha) = \abs{d} \sin(\alpha).
\]
The sine function is infinitesimal if and only of $\alpha \simeq 0$ or $\alpha \simeq \pi$, which can be easily seen by the series expansion. This is the claim.
\end{proof}
\begin{lemma}[Cross-Product and Almost Equivalent]
For $x,y \in \CsP^2$ the appreciable cross-product $x \times_* y$ is infinitesimal if and only if $x\simeq y$.
\end{lemma}
\begin{proof}
   Let $x \simeq y$ then there exists $\lambda \in \Cs \setminus \I$ such that $x_\A \simeq \lambda \cdot y_\A$. Then it follows:
\begin{align*}
    x \times_* y &= x_\A \times y_\A \simeq \lambda \cdot y_\A \times y_\A = 0 \\  &\Rightarrow  x \times_* y \in \I
\end{align*}
   The other direction of the proof follows from \autoref{nsa_pg:aeq_angles}: two vectors are almost linearly dependent, if the angle between the vectors is either infinitesimal or infinitely close to $\pi$, thus if their appreciable cross-product is an infinitesimal vector.
\end{proof}
The next example shows that we cannot neglect the representative of an equivalence class for the cross-product.
\begin{example}[Appreciable Cross-Product]
    \label{nsa_pg:appr_cross_expl}
   Let $H \in \Rs_\infty$ and define  
   \[
       x = \begin{pmatrix} H\\0\\1\end{pmatrix}, \quad y = \begin{pmatrix} 1\\0\\0\end{pmatrix}
   \]
   Both represent the same equivalence class in $\RP^2$ which we see by their projective shadows 
   \[
   \psh( x) = \sh \begin{pmatrix} 1\\0\\\frac{1}{H}\end{pmatrix} = \begin{pmatrix} 1\\0\\0\end{pmatrix}, \quad \psh(y) = y = \begin{pmatrix} 1\\0\\0\end{pmatrix}
   \]
   So the unnormalized cross-product should resemble the zero vector, but as we can see
   \[
   x \times y = \begin{pmatrix} 0\\1\\0\end{pmatrix}, \quad  x \times_{\sh} y = \sh \begin{pmatrix} 0\\ \frac{1}{H} \\ 0 \end{pmatrix} = \begin{pmatrix} 0\\ 0 \\ 0 \end{pmatrix}.
   \]
the cross-product yields a wrong result, while the shadow cross-product is correct.
\end{example}
Now we will come back to relations of points and lines. One important concept is collinearity of three points. We will generalize this to a further ``almost'' relation.
\begin{definition}[Almost Collinear]
    \label{nsa_pg:almost_collinear}
    We call three points $x,y,z \in \CsP^2$ \textit{almost collinear}, 
    if the point $x$ is almost incident to the line $l:=\join_*(y,z)$, \ie 
    \[
        \langle x, l \rangle_* = \langle x_\A, l_\A \rangle \in \I.
    \]
\end{definition}
\begin{lemma}[Almost Collinear is Well-defined] 
    The definition of being almost collinear in \autoref{nsa_pg:almost_collinear} is well defined.
\end{lemma}
\begin{proof}
    Special case of \autoref{nsa_pg:i_well_def}.
\end{proof}
The canonical way to check for collinearity is usually the determinant, which we will now define in an normalized version.
\begin{definition}[Normalized Determinant]
    \label{nsa_pg:norm_det}
    Let $x,y,z$ be in three distinct objects in $\CsP^2$. We define the \textit{normalized determinant of $\mathit{x,y,z}$} by
    \[
        \det_{\normsymb}[x,y,z] = \frac{\det[x_\A, y_\A, z_\A]}{\lambda}, \quad \lambda \in \magni{(\norm{y_\A \times z_\A})}
        \]
        where $x_\A, y_\A, z_\A$ are appreciable representatives of $x,y,z$ and $\det$ the standard definition of the determinant.
\end{definition}
\begin{remark}
   The formula might look a bit asymmetric but term we are dividing by assures that the cross-product of $y_\A$ and $z_\A$ is again appreciable.
\end{remark}
\begin{lemma}[Determinats, Scalar-Product and Cross-Product]
    \label{nsa_pg:det_scal_cross}
   For  $x,y,z \in \CsP^2$ and fixed appreciable representatives of $x_\A, y_\A, z_\A$ of $x,y,z$. 
   Then it holds true:
   \[
       \det_{\normsymb}[x,y,z] = \langle x, y \times_* z \rangle_*
   \]
   where $\langle \cdot, \cdot \rangle_*$ denotes the appreciable scalar product and $(\cdot \times_* \cdot)$ the appreciable cross product.
\end{lemma}
\begin{proof}
    Since \[
        \det_{\normsymb}[x,y,z] = \frac{\det[x_\A, y_\A, z_\A]}{\lambda}, \quad \lambda \in \magni{(\norm{y_\A \times z_\A})}
    \]
        we can rewrite the latter, using the rewrite rule for the ``normal'' determinant, to
    \[
        \det[x_\A, y_\A, z_\A] = \langle x_\A,  y_\A \times z_\A \rangle     
    \]
    Then it holds true for $\lambda \in \magni{(\norm{y_\A \times z_\A})}$:
    \begin{align*}
        \det_*[x,y,z] &= \frac{\det[x_\A, y_\A, z_\A]}{\lambda} = \frac{\langle x_\A,  y_\A \times z_\A \rangle }{\lambda}\\ &= \langle x_\A, \frac{ y_\A \times z_\A}{{\lambda}} \rangle = \langle x, y \times_* z \rangle_*.
    \end{align*}
   \end{proof}
   \begin{corollary}[Collinearity and the Normalized Determinat]
      The distinct points $x,y,z$ in $\CsP^2$ are almost are almost collinear if and only if 
      \[
     \det_{\normsymb}[x,y,z] \in \I.
    \]
   \end{corollary}
   \begin{proof}
By \autoref{nsa_pg:det_scal_cross} and the definition of being almost collinear.
   \end{proof}
\begin{definition}[Appreciable Determinant]
    \label{nsa_pg:appr_det}
    Let $x,y,z$ be in three distinct vectors in $\CsP^2$. We define the \textit{appreciable determinant of $\mathit{x,y,z}$} by
    \[
        \det_{*}[x,y,z] := \det[x_\A, y_\A, z_\A]
        \]
        where $x_\A, y_\A, z_\A$ are appreciable representatives of $x,y,z$ and $\det$ the standard definition of the determinant.
\end{definition}
\begin{lemma}[Appreciable Determinant and Almost Collinear]
    \label{nsa_pg:appr_det_ac}
    Let $x,y,z$ be three points in $\CsP^2$ which are pairwise not almost equivalent. Then $x,y,z$ are almost collinear if and only if $\displaystyle{\det_{*}[x,y,z] \in \I}$.
\end{lemma}
\begin{proof}
    Special case of the proof of \autoref{nsa_pg:det_scal_cross}. The lack of the normalization using $\lambda \in \magni{(\norm{y_\A \times z_\A})}$ is negligible since if $y$ and $z$ are not almost equivalent then their cross product yields an appreciable vector.
\end{proof}
       One might wonder why we defined two different types of determinants: the appreciable and the normalized one. Both have their purpose and are can be used to determine if points are almost collinear, but the appreciable determinant might give false positives if two (or more) points are almost equivalent. We will illustrate this with an example:
   \begin{example}[Almost Collinear and Normalized Determinants]
      Pick an arbitrary $\epsilon \in \I$ and define the points
      \[
          x = \begin{pmatrix}
             1 \\ 0 \\1 
          \end{pmatrix}\quad 
          y = \begin{pmatrix}
             \epsilon \\ 0 \\1 
          \end{pmatrix} \quad
          z = \begin{pmatrix}
             0 \\ \epsilon  \\1 
          \end{pmatrix}
      \]
      These points are \textit{not} almost collinear, the join of $y$ and $z$ is the line $(\epsilon, \epsilon, -\epsilon^2)^T \sim (1,1, -\epsilon)^T$, which is almost equivalent to the angle bisector of $x$- and $y$-axis. In particular, observe that the appreciable join of $y$ and $z$ is infinitesimal and so $y$ and $z$ are almost equivalent (see \autoref{nsa_pg:aeq_angles}).

      The appreciable determinant gives a false positive incidence relation here: first note that $x,y,z$ are already appreciable and the appreciable determinant yields $\det_*[x,y,z] = -\epsilon(1-\epsilon)$, which is infinitesimal. This is not a contradiction to \autoref{nsa_pg:appr_det_ac} since $y$ and $z$ are almost equivalent. This resembles the fact that $y$ and $z$ have the same shadow and therefore the join of their shadows is not an element of the (standard) projective space.

      The normalized determinant yields the correct result:  $\det_{\normsymb}[x,y,z] =\frac{\det_*[x,y,z]}{\epsilon} = 1-\epsilon \not \in \I$ where we used that $\epsilon \in \magni(\norm{y \times z})$. This also corresponds to the definition of being almost incident:
      \[
          \langle x, l \rangle_* = \langle x_\A, l_\A \rangle = 1 - \epsilon \not \in \I.
      \]
   \end{example}
\section{Non-standard Projective Transformations}
This section will start out with basic properties of linear algebra over a hyper field. Since we are interested in a projective setting, we will directly identify a matrix $M$ with its appreciable representative, \ie its normalized version $\frac {M}{\norm{M}}$.
\begin{definition}[Appreciable Matrix]
    \label{nsa_pg:appr_matrix}
    Let $M \in \Cs^{m \times n}$ and not all entries of $M$ equal zero. Then we define an \textit{appreciable matrix representation $\mathit{M_\A}$ of M} by 
    \[
        M_\A:= \dfrac 1 \lambda M, \quad \lambda \in \magni{(\norm{A})} 
    \]
    where $\norm{\cdot}$ is defined as an arbitrary submultiplicative and self-adjoint matrix norm. If not stated otherwise we will use, for convenience, the Frobenius norm:
    \[
        \| M \|_F := \sqrt{\sum_{i=1}^m \sum_{j=1}^n |M_{i,j} |^2}.
    \]
\end{definition}
\begin{remark}
    It is important to have a submultiplicative matrix norm in order to use estimations like $\norm{Ax} \leq \norm{A}\norm{x}$. All important matrix norms are submultiplicative, \eg the spectral norm, natural norm, Frobenius norm or the (scaled) maximum norm. Note that only taking the absolute value of the absolute maximum of a matrix it \textit{not} submultiplicative! The self-adjoint property is also important: then transposition (and complex conjugation) do not change the matrix norm.
\end{remark}
\begin{lemma}[Appreciable Matrix Properties]
    Let $M_\A$ be an appreciable matrix. Then all entries of $M_\A$ are limited and at least one is appreciable.
\end{lemma}
\begin{proof}
    Analogously to the proof of \autoref{nsa_pg:lim_entry}.
\end{proof}
Next we will define a conic section, or shortly conic, in the same way as J.\ Richter-Gebert's perpectives on projective geometry \cite{richter2011perspectives} \page 145 ff. Our conics are enlarged versions of the original definitions.
\begin{lemma}[Appreciable Matrix-Vector Product]
    \label{nsa_pg:appr_mv}
    Let $M_\A \in \Cs^{m \times n}$ be an appreciable matrix and $x_\A \in\CP^{n-1}$ an appreciable representative of $x \in \CP^{n-1}$. Then their the matrix--vector product 
    \[
        M_\A \cdot x_\A
    \]
    is limited.
\end{lemma}
\begin{proof}
    All entries of $M_\A$ and $x_\A$ are appreciable, products and sums of appreciable numbers are limited according to \autoref{nsa_basics:arith}.
\end{proof}
\begin{definition}[Almost Singular]
    We call a matrix $M \in \Ks^{3\times 3} \Smz$ \textit{almost singular} it the determinant of its appreciable representative $M_\A = \frac{1}{\lambda} M$, $\lambda \in \magni{(\norm{M})}$ is infinitesimal but not zero, \ie
    \[
        \det(M_\A) \in \I \Smz.
    \]
    If the $\det(M_\A) = 0$, we call $M$ \textit{singular}, and \textit{non-singular} if $\det(M_\A)$ is appreciable.
\end{definition}
\begin{remark}
    If $M \in \mathrm{GL}_3(\Ks)$ then $M$ non-singular or almost singular. One can easily see that if one applies the Gaussian algorithm to decompose $M$ into an upper and lower triangular matrix and uses the equivalence of full rank and $\det(\cdot) \neq 0$.
\end{remark}
\begin{remark}
    A.\ Leitner chose a different approach for (in our terms) non-singular projective transformations: she called a hyperreal projective transformations $A \in  \operatorname{PGL}_n(\Rs)$ (where $\operatorname{PGL}$ denotes the projective linear group) \textit{finite} if there is a standard real projective transformation $B \in  \operatorname{PGL}_n(\R)$ and a $\lambda \in \Rs$ such that $B - \lambda A$ is infinitesimal \cite{leitner2015limits}. This is equivalent to our definition of non-singular matrices: if a projective transformation is finite, it holds true that:
    \[
        B \simeq \lambda A \Rightarrow \det(B) \simeq \det(\lambda A)
    \]
    by the continuity of the determinant. We know that $\det(B) \in \R \Smz$, so it holds true that $\det(\lambda A) \in \A$. And by the same arguments as in \autoref{nsa_pg:mag_appr} we know that $\lambda^{-1} \in \magni(\norm{A})$ and therefore $\lambda A$ is non-singular 
    
    On the other hand, for a non-singular projective transformation $M$ we know that the determinant of an appreciable representative $M_\A$ is appreciable. Then it holds true that the shadow of $M_\A$ is a standard projective transformation. So the objects $B:=\sh(M_\A), \lambda = \frac{1}{\norm{M}}$, the matrix $B - \lambda A$ is infinitesimal, which is Leitners definition. 
\end{remark}
\begin{theorem}[Appreciable Image]
    \label{nsa_pg:appr_image}
    Let $M \in \mathrm{GL}_3(\Ks)$, let $M$ be regular and $p \in \Ks^3 \Smz$. Then the matrix vector multiplication of $M_\A = \frac{1}{\norm{M}} M$ and $p_\A = \frac{1}{\norm{p}} p$ is an appreciable vector, \ie $\norm{M_\A p_\A} \in \A$.
\end{theorem}
\begin{proof}
    We know that all entries of both $M_\A$ and $p_\A$ are limited. The matrix-vector multiplication consists of multiplication and summation, which preserve the property of being limited (see \autoref{nsa_basics:arith}). So the norm of $M_\A p_\A$ cannot be unlimited. 

    $M_\A p_\A$ cannot be an infinitesimal vector either. Assume the contrary, then we can use \autoref{nsa_basics:sh_prop} and find
    \[
        0 = \sh(M_\A p_\A) = \sh(M_\A) \sh(p_\A).
    \]

    Since $p_\A$ has appreciable length, it holds true that $\sh(p_\A) \in \K^3 \Smz$. But this means that $\sh(p_\A)$ is an element of the kernel of $\sh(M_\A)$ and so $\sh(M_\A)$ is singular. But this implies that either $M$ is singular or almost singular, which is a contradiction.
\end{proof}
\begin{definition}[$\epsilon$-Kernel]
    Let $M \in \Ks^{3 \times 3}$. We define the \textit{$\mathit{\epsilon}$-kernel} by 
    \[
        \epsilon(M):= \{ v \in \CsP^2 \bbar M_\A v_\A \simeq 0\}
    \]
where $M_\A$ and $p_\A$ are appreciable representatives of $M$ and $p$.
\end{definition}
\begin{remark}
    By \autoref{nsa_pg:appr_image} we know that for regular matrices the $\epsilon$-kernel consists only of vectors which are almost equivalent to the zero vector, which are not part of a projective space.
\end{remark}
We define the adjugate of a matrix in the usual way:
\begin{definition}[Adjugate non--standard matrix]
    For an appreciable matrix $M_\A \in \Cs^{3 \times 3}$ with 
    \[
        M_\A = \begin{pmatrix}a & b & c\\ d & e & f \\ g & h & i\end{pmatrix}
    \]
    we define the \textit{(non-standard) adjugate matrix $\mathit{M_\A^\triangle}$} by
    \[
        \mathit{M_\A^\triangle}:=
\begin{pmatrix}
\quad\det\begin{pmatrix}e & f\\ h & i\end{pmatrix} &
- \det\begin{pmatrix}d & f\\ g & i\end{pmatrix} &
\quad\det\begin{pmatrix}d & e\\ g & h\end{pmatrix} \\
- \det\begin{pmatrix}b & c\\ h & i\end{pmatrix} &
\quad\det\begin{pmatrix}a & c\\ g & i\end{pmatrix} &
- \det\begin{pmatrix}a & b\\ g & h\end{pmatrix} \\
\quad\det\begin{pmatrix}b & c\\ e & f\end{pmatrix} &
- \det\begin{pmatrix}a & c\\ d & f\end{pmatrix} &
\quad\det\begin{pmatrix}a & b\\ d & e\end{pmatrix}
\end{pmatrix}^T
    \]
    where the determinant $\det$ is defined in the usual way, just with non--standard arithmetic.
\end{definition}
\begin{lemma}[Adjugate Non-standard Matrix is Limited]
    \label{nsa_pg:lem:adju_appr}
Let $M_\A$ be an appreciable matrix. The entries of the adjugate matrix 
$M_\A^\triangle$ are limited.
\end{lemma}
\begin{proof}
  A.\ Leitner gave a similar proof in \cite{leitner2015limits} for the real case. The entries of $M_\A^\triangle$ are generated by multiplication and substraction of appreciable numbers. This yields limited (if the sum is infinitesimal) or appreciable results by \autoref{nsa_basics:arith}.
\end{proof}
\begin{lemma}[Norm of the Inverse]
    \label{nsa_pa:norm_inverse}
    Let $M \in \mathrm{GL}_3(\Ks)$ and let $M$ be regular. Then it holds true that $\norm{M_\A^{-1}}$ is appreciable. 
\end{lemma}
\begin{proof}
   It is obvious that $\norm{M_\A}$ has appreciable norm just by definition. Since $\norm{\cdot}$ is submultiplicative we know that there is a $c \in \R \Smz$ such that \[
        1 = \norm{\Id} = \norm{M_\A (M_\A)^{-1}} \leq \norm{M_\A} \norm{ (M_\A)^{-1}}.
    \]
    Then we know that $\frac{1}{\norm{M_\A}} \leq  \norm{ (M_\A)^{-1}}$ and since $\norm{M_\A}$ is appreciable its inverse is also appreciable and so $ \norm{ (M_\A)^{-1}}$ cannot be infinitesimal.

    We still have to show that the norm of $(M_\A)^{-1}$ is not unlimited: 
    \begin{align*}
        (M_\A)^{-1} &= \frac{1}{\det(M_\A)} (M_\A)^\triangle \\
        \Rightarrow  \norm{(M_\A)^{-1}} &= \frac{1}{\abs{\det(M_\A)}} \norm{ (M_\A)^\triangle} 
    \end{align*}
    By \autoref{nsa_pg:lem:adju_appr} we know that all entries of the adjugate matrix $M_\A^\triangle$ are appreciable and furthermore ${\abs{\det(M_\A)}}^{-1}$ is appreciable since $M$ is non-singular. The product of an appreciable and a limited number is limited.

Combining both arguments the norm of the inverse has to be appreciable.
\end{proof}
Completely analogously to projective geometry over $\R$ or $\C$ we define a projective transformation.
\begin{definition}[Appreciable Projective Transformation]
    Let $M \in \mathrm{GL}_3(\Ks)$. For $p \in \PCs$ and $l \in \LCs$ we define the \textit{appreciable projective transformation by $\mathit{M}$} as follows: 
    \[
      p \mapsto  M_\A p_\A \quad \text{and} \quad l \mapsto (M_\A)^{-H} l_\A
    \]
    with $M^{-H}:= (\overline{M^{-1}})^T$. We call $M_\A$ an \textit{appreciable transformation matrix}. If $M$ is non-singular, we call the projective transformation \textit{non-singular} as well.
\end{definition}
\begin{theorem}[Non-singular Projective Transformations and Almost Equivalent]
   Let $M$ be a non-singular projective transformation and let $p, p' \in \PCs$ be almost equivalent. Then the images of $p$ and $p'$ under the regular projective transformation $M$ are also almost equivalent. 
\end{theorem}
\begin{proof}
   We have to show that 
   \[
       \frac{1}{\norm{M_\A p_\A}} M_\A p_\A \times \frac{1}{\norm{M_\A p'_\A}} M_\A p'_\A \in \I^3.
   \]
   Equivalently we show that the norm of the cross-product above is infinitesimal:
   \begin{align*}
       \frac{1}{\norm{M_\A p_\A}} M_\A p_\A \times \frac{1}{\norm{M_\A p'_\A}} M_\A p'_\A &= \norm{\frac{{\det(M_\A)} }{ \norm{M_\A p_\A}\norm{M_\A p'_\A} } {M_\A^{-T}} \left( p_\A \times p'_\A \right) } \\  &\leq
   \frac{\abs{\det(M_\A)}}{ \norm{M_\A p_\A}\norm{M_\A p'_\A} }\norm{M_\A^{-T}} \norm{ p_\A \times p'_\A } \\
   &=  \frac{\abs{\det(M_\A)}}{ \norm{M_\A p_\A}\norm{M_\A p'_\A} }\norm{M_\A^{-1}} \norm{ p_\A \times p'_\A }
   \end{align*}
   Where we used that we chose $\norm{\cdot}$ submultiplicative in the inequality and that 
   transposition do not change the norm of a matrix for self-adjoint matrix norms. By assumption $M$ is regular and therefore the determinant of $M_\A$ is appreciable (and so is its absolute value). By \autoref{nsa_pg:appr_image}, we know that the terms $\norm{M_\A p_\A}, \norm{M_\A p'_\A}$ are appreciable and so are their inverse fractions. Then by \autoref{nsa_pa:norm_inverse}, we know that the norm of the inverse of $M_\A$ is also appreciable. So the first two terms yield an appreciable result. Finally, by assumption we know that $\norm{ p_\A \times p'_\A }$ is infinitesimal, which multiplied with an appreciable number is infinitesimal, which proofs the claim.
\end{proof}
\begin{theorem}[Regular Projective Transformations and Almost Incident]
    \label{nsa_pg:reg_trafo_incident}
   Let $l \in \PCs$ and $p \in \LCs$ be almost incident and let $M$ be an non-singular projective transformation. Then the images of $l$ and $p$ under $M$ are also almost incident.
\end{theorem}
\begin{proof}
   Since $l$ and $p$ are almost incident there is an $\epsilon \in \I$ such that 
   \begin{align*}
       &\epsilon = \langle l, p \rangle_* = \langle l_\A, p_\A \rangle = \langle (M_\A)^{-H} l_\A, M_\A p_\A \rangle  \\
                \Leftrightarrow &\epsilon \cdot \norm{M_\A p_\A} \norm{(M_\A)^{-H} l_\A}  = \langle \frac{(M_\A)^{-H} l_\A}{\norm{(M_\A)^{-H}l_\A}}, \frac{M_\A p_\A}{\norm{M_\A p_\A}} \rangle
   \end{align*}
   By \autoref{nsa_pg:appr_image}, we know that $M_\A p_\A$ is appreciable and so is its norm. 
   
We chose $M$ to be a regular matrix and so the inverse of $M_\A$ exists and its matrix norm is appreciable by \autoref{nsa_pa:norm_inverse}. Analogously to the proof of \autoref{nsa_pg:appr_image} we can bound the term $\norm{(M_\A)^{-H} l_\A} \leq  \norm{(M_\A)^{-H}} \norm{ l_\A} =  \norm{(M_\A)^{-1}} \norm{ l_\A}$ which is a product of appreciable numbers and therefore appreciable. Then $\epsilon':= \epsilon \cdot  \norm{(M_\A)^{-H}}\norm{M_\A p_\A}$ is infinitesimal as product of an infinitesimal and two appreciable numbers. Then, we can conclude for arbitrary $c, c' \in \A$
\[
    \Leftrightarrow \epsilon'':= \epsilon \cdot c \cdot c'= \langle \frac{(M_\A)^{-H} l_\A}{c' \norm{(M_\A)^{-H} l_\A}}, \frac{M_\A p_\A}{c \norm{M_\A p_\A}} \rangle = \langle M^{-H} l, M p \rangle_*
\]
which shows that the appreciable projective transformation of $l$ and $p$ are almost incident, that was the claim.
\end{proof}
\begin{definition}[Almost Affine Projective Transformation]
    \label{nsa_pg:def:almost_affine}
    Let $M \in \mathbb{GL}_3(\Rs)$ be a projective transformation. We call $M$ \textit{almost affine}, if there is an appreciable representative $M_\A$ of $M$ which can be written as 
   \[
M_\A =       \begin{pmatrix}
           c & s & a \\ -s & c & b \\ \epsilon & \delta &1
       \end{pmatrix}
   \]
 with $\epsilon, \delta \in \I$ and $c^2 + s^2 \not \in \I$.
\end{definition}
\begin{lemma}[Almost Collinear and Regular Projective Transformations]
   Let $p,q,r \in \PCs$ be three distinct points almost conlinear points in $\RsP^2$ and let $M$ be a regular non-standard projective transformation. Then it holds true that the transformed points $Mp, Mq, Mr$ are also almost collinear.
\end{lemma}
\begin{proof}
    Consider the line $l:= \join(p,q)$ and apply \autoref{nsa_pg:reg_trafo_incident} to $l$ and $r$.
\end{proof}
\begin{lemma}[Almost Singular Transformations and Almost Relations]
    There are almost singular projective transformations which do \textit{not} preserve the property of being almost equivalent and almost collinear.
\end{lemma}
\begin{proof}
    By example: let $\epsilon$ be a positive hyperreal and define the appreciable projective transformation (which is already an appreciable representative)
   \[
   M = \begin{pmatrix}
       1 & 0 & 0\\ 0& 1 & 0\\ 0& 0& \epsilon 
   \end{pmatrix}.
   \]
   Then $\det(M) = \epsilon \in \I$, so $M$ is almost singular. Pick the following points $a,b,c \in \PCs$:
   \[
       a = \begin{pmatrix}
           0 \\ 0 \\ 1 
       \end{pmatrix} \quad
       b = \begin{pmatrix}
           \epsilon \\ 0 \\ 1 
       \end{pmatrix} \quad
       c = \begin{pmatrix}
           0 \\ \epsilon \\ 1 
       \end{pmatrix} 
   \]
   It's easy to see that the points are almost equivalent. Additionally the connecting line $l$ of $a$ and $b$ is almost incident to $c$: 
   \[
       \join_*(a,b) = \frac{a \times b}{\norm{a\times b}} =  \begin{pmatrix}
           0 \\ 1 \\ 0 
       \end{pmatrix} \Rightarrow \langle c, l \rangle_* = 0 + \epsilon + 0 = \epsilon \in \I
   \]
   But the appreciable projective transformations of $a,b$ and $c$ are
   \begin{align*}
      M a =  \begin{pmatrix}
           0 \\ 0 \\ \epsilon
       \end{pmatrix} \sim \begin{pmatrix}
           0 \\ 0 \\ 1       \end{pmatrix}, \quad
      M b =  \begin{pmatrix}
           \epsilon \\ 0 \\ \epsilon
       \end{pmatrix} \sim \begin{pmatrix}
           1 \\ 0 \\ 1       \end{pmatrix}, \quad
M c =  \begin{pmatrix}
            0 \\ \epsilon \\ \epsilon 
       \end{pmatrix} \sim \begin{pmatrix}
           0 \\ 1 \\ 1       \end{pmatrix} 
   \end{align*}
   which are obviously neither almost collinear nor almost equivalent!
\end{proof}
\begin{remark}
    This is quite remarkable since linear functions as they are described by matrices are continuous everywhere and so one would expect the transformations to preserve ``almost''-relations. The crucial point here is not the transformation itself but the normalization afterwards: the function $\frac{ x}{ \norm{x}}$ is undefined for zero vector and is discontinuous for infinitesimal vectors and this is where the points are ``pushed apart''. 

    The points $a,b$ and $c$ in the previous statement were chosen such that they are in the $\epsilon$-kernel of $M$. Hence, $Ma$, $Mb$ and $Mc$ yield vectors with infinitesimal length that are after normalization not almost equivalent anymore.
\end{remark}
\section{Non-standard Cross-Ratios}
We will start with the basic definition of a cross-ratio in $\Ks^2$.
\begin{definition}[Cross-Ratio in $\Ks^2$]
    For $A,B,C,D \in \Ks^2$ we define the cross-ratio in the standard way (see \eg ``Geometriekalküle'' \cite{richter2009geometriekalkule} \page 42).
   \[
       (A,B;C,D) := \frac{[A,C] [B,D]}{[A,D] [B,C]}
   \]
\end{definition}
\begin{remark}
    It is a bit surprising that the cross-ratio is analogously defined as in the standard setting, one would assume that due the appearance of infinitesimal and unlimited values a normalization would be necessary (for example the employment of the appreciable or the normalized determinant). But the cross-ratio is completely independent of the representative of a vector as we will see in the next statement.
\end{remark}
\begin{lemma}[Cross-Ratios and Representatives]
   Let $\lambda_A, \lambda_B, \lambda_C, \lambda_D \in \Ks \Smz$. Then it holds true that
   \[
       (A,B; C, D) = (\lambda_A \cdot A, \lambda_B \cdot  B; \lambda_C \cdot  C, \lambda_D \cdot  D).
   \]
\end{lemma}
\begin{proof}
    Completely analogously to ``Geometriekalküle'' \cite{richter2009geometriekalkule} \page 42.
\end{proof}
\begin{remark}
   Particularly interesting is the fact that we did only exclude $0$ from the possible scaling factors of $A,B,C,D$ and not all infinitesimal numbers. The reason for this is that the cross-ratio is completely independent of the representative even for non-Archimedean fields.
\end{remark}
\begin{lemma}[Projective Transformations and Cross-Ratios]
    Let $M \in \Ks^{3\times 3}$ be regular or almost singular then 
    \[
        (A,B; C,D) = (M \cdot A, M \cdot B; M \cdot C; M \cdot D).
    \]
\end{lemma}
\begin{proof}
Similar to ``Geometriekalküle'' \cite{richter2009geometriekalkule} \page 42.
\end{proof}
\begin{remark}
   Again it is at least a bit astonishing that even if a projective transformation is almost singular it still preserves cross-ratios. 
\end{remark}
\begin{theorem}[Almost Equivalent Cross-Ratios]
   Let $A, B, C, D$ and $A', B', C', D'$ be vectors with $A \simeq A', B \simeq B', C \simeq C', D \simeq D'$ then it holds true:
   \[
       \sh((A,B;C,D)) = \sh((A',B',C',D'))
   \]
\end{theorem}
\begin{proof}
    By the continuity of the determinant.
\end{proof}
\begin{remark}
    One has to be careful with the relation above: it does \textit{not} hold true that 
    \[
(A,B;C,D)) \simeq ((A',B',C',D').
        \]
        Take for example three disjoint points $A,B,C$ and consider the cross-ratio $R:=(A,B;C,D)$ for $D=A$. The value $R$ of this cross-ratio is $\infty$ since we divide by zero. If one now takes $D \simeq A, D \neq A$ then the cross-ratio $r$ will be unlimited, but $R \not \simeq \infty$.
\end{remark}
\section{Non-standard Conics}
\begin{definition}[Non-standard Conic]
    \label{nsa_pg:nst_conic}
    For a given matrix $M \in \Cs^{3 \times 3}$ whose entries are not all equal to zero. We define a \textit{non-standard conic with associated matrix $\mathit{{M}}$} by 
    \[
        \mathcal C_{M} := \{ \, [p] \in  \mathcal{P}_\Cs \bbar p_\A^T M_\A p_\A  \, \simeq  \, 0 \}
    \]
    where $p_\A$ denotes an appreciable representative of $p$ and $M_\A$ is an appreciable matrix representation of $M$.
\end{definition}
\begin{definition}[Point and Conic Almost Incidence Relation]
    \label{nsa_pg:pc_inci}
    Let a point $ [p] \in  \mathcal{P}_\Cs$ and a non--standard conic $\mathcal C_{M}$ with associated matrix $\mathit{{M}}$ fulfill the relation of \autoref{nsa_pg:nst_conic}: $ p_\A^T M_\A p_\A  \, \simeq  \, 0$. Then we call $p$ and $\mathcal C_{M}$ \textit{almost incident} and write 
    \[
    [p] \, \mathcal{I}_\Cs \, [\mathcal C_M]    
    \]
\end{definition}
\begin{remark}
    Of course the standard incidence relation of a point and a conic is a special case of the almost incidence relation due to the fact that $0$ is infinitesimal (and also the only infinitesimal value in $\C$). 
\end{remark}
\begin{lemma}[Well-defined Conic]
    The almost incidence relation of points and conics in \autoref{nsa_pg:pc_inci} is well defined.
\end{lemma}
\begin{proof}
   We have to show that the relation $ p_\A^T M_\A p_\A  \, \simeq  \, 0$ is well defined. Define $l_\A:= M_\A p_\A$, then  
   \[
        p_\A^T M_\A p_\A  = \langle p_\A, M_\A p_\A \rangle = \langle p_\A, l_\A \rangle 
   \]
   We can interpret $l_\A$ as element of $\mathcal{L}_\Cs$ (the polar) which is appreciable by \autoref{nsa_pg:appr_mv}. Use \autoref{nsa_pg:i_well_def} where we showed that the property of being infinitesimal in the appreciable scalar product is well defined. This is the claim.
\end{proof}
\begin{lemma}[Conics and Projective Halo]
   Let $\mathcal C$ be a non-standard conic and $p \in \CP^2$ be incident to $C$. Then a point $p'$ which almost equivalent $p' \simeq p$ is almost incident to $\mathcal C$.

   In other words: the projective halo $\phal(p)$ is almost incident to $C$.
\end{lemma}
\begin{proof}
    Let $M \in \Cs^{3 \times 3}$ be the associated matrix to $C$. Since $p'$ is almost equivalent to $p$ there is $\lambda \in \A$ such that $p_\A = \lambda \cdot p'_\A$ and due to the incidence of $p$ and $\mathcal C$ it holds true:
    \[
        0 = p^T M_\A p \simeq \lambda^2 p'^T M_\A p'.    \]
    Since $\lambda^2 \in \A$ it holds true that $ p'^T M_\A p' \in \I$ and thus $  p' \, \mathcal{I}_\Cs \, \mathcal C$.
\end{proof}
We will now analyze cocircularity and therefore we need to define two special points of $\CP^2$.
\begin{definition}[The Points $\II$ and $\JJ$, \cite{richter2011perspectives} \page 330]
   We define the points $\II$ and $\JJ$ in $\CP^2$ by
   \[
       \II := \begin{pmatrix}
          -\ii \\ 1 \\ 0 
       \end{pmatrix}
       \quad \text{and} \quad
       \JJ := \begin{pmatrix}
          \ii \\ 1 \\ 0 
       \end{pmatrix}.
   \]
\end{definition}
\begin{remark}
As proven in ``Perspectives on Projective Geometry'' (\cite{richter2011perspectives} \page 330 ff.) the points $\II$ and $\JJ$ are incident to all circles and a conic section is a circle if it passes through $\II$ and $\JJ$.
\end{remark}
\begin{definition}[Almost Cocircular]
    We call four points $A,B,C,D \in \RsP^2$ \textit{almost cocircular}, if $\JJ$ is almost incident to the conic section defined by $A,B,C,D,\II$.
\end{definition}
\begin{theorem}[Cocircluarity]
   Four points $A,B,C,D \in \RsP^2$ are almost cocircular, if and only if the following relations holds true:
   \[
       [CA\II][DB\II][DA\JJ][CB\JJ] - [CA\JJ][DB\JJ][DA\II][CB\II] \simeq 0.
   \]
\end{theorem}
\begin{proof}
    Essentially the same proof as in \cite{richter2011perspectives} page~331 ff.\ if one replaces $=$ with $\simeq$.
\end{proof}
\begin{theorem}[Almost Cocircular and Almost Affine Transformations]
    An almost affine projective transformation preserves the almost cocircular property and all non-singular projective transformations that leave $\II'$ and $\JJ'$, with $\II \simeq \II'$ and $\JJ' \simeq \JJ$, almost invariant (\ie $I' \simeq M\II', \JJ' \simeq M\JJ'$) are almost affine.
\end{theorem}
\begin{proof}
    By \autoref{nsa_pg:def:almost_affine} a almost affine projective transformation $M$ has an appreciable representative of the form
   \[
     M =  \begin{pmatrix}
           c & s & a \\ -s & c & b \\ \epsilon & \delta &1
       \end{pmatrix}
   \]
   with $\epsilon, \delta \in \I$ and $c^2 + s^2 \not \in \I$. Let $I'$ be infinitely close to $\II$ which means there is $\lambda \in \A$ and $\Tau:= (\tau_1, \tau_2, \tau_3)^T \in \I^3$ such that 
   \[
       \II \simeq \II' \Leftrightarrow \II' = \lambda \begin{pmatrix}
          -\ii \\ 1 \\ 0 
       \end{pmatrix} + \begin{pmatrix}
       \tau_1 \\ \tau_2 \\ \tau_3        \end{pmatrix}  
   \]
   By \autoref{nsa_pg:appr_image} we know that the product of $\lambda  M \cdot I$ is appreciable and the product of $ M \cdot \Tau$ is infinitesimal, so $\lambda M \cdot \II + M \cdot \Tau$ almost equivalent to $\lambda M \cdot \II$, then we find
\begin{align*}
    M \cdot \II' &=       \begin{pmatrix}
           c & s & a \\ -s & c & b \\ \epsilon & \delta &1
       \end{pmatrix}  \left(
\lambda\begin{pmatrix}           -\ii \\ 1 \\ 0 
       \end{pmatrix} + \begin{pmatrix}
   \tau_1 \\ \tau_2 \\ \tau_3        \end{pmatrix} \right)   
   \\ &\simeq \lambda \begin{pmatrix}
           c & s & a \\ -s & c & b \\ \epsilon & \delta &1
       \end{pmatrix}  \begin{pmatrix}           -\ii \\ 1 \\ 0 
       \end{pmatrix} = 
       \begin{pmatrix}
      -\ii \cdot c +s \\ \ii \cdot s + c   \\ -\epsilon \cdot \ii + \delta
  \end{pmatrix} \simeq  
       \begin{pmatrix}
       -\ii \cdot c +s \\ \ii \cdot s + c   \\ 0       \end{pmatrix} = (c+ \ii \cdot s)  \II \sim \II
\end{align*}
Analogously we can show the same property for $\JJ$.

Conversely let $M$ be a matrix with the property $\II \simeq M \cdot \II'$. We show the claim analogously to ``Perspectives on Projective Geometry'' \cite{richter2011perspectives} \page 336 ff. We start with the first two entries of the last row of $M_\A$: we claim that they are infinitesimal. Remember that all entries of an appreciable matrix are limited and at least one is appreciable.
\[
    \begin{pmatrix}
        \bullet & \bullet & \bullet \\
        \bullet & \bullet & \bullet \\
        x & y & \bullet
    \end{pmatrix}
    \left(
    \begin{pmatrix}
        -\ii  \\ 1 \\ 0
    \end{pmatrix}
    +
    \begin{pmatrix}
        \epsilon_x  \\ \epsilon_y \\ \epsilon_z
    \end{pmatrix}
    \right)
    \simeq
    \lambda
    \begin{pmatrix}
        -\ii  \\ 1 \\ 0
    \end{pmatrix}
\]
This means that $-\ii x + y \simeq 0$ but since $x,y $ are real and limited this means that both entries have to be infinitesimal. The rest of the argumentation is analogously to the one in the mentioned source.
\end{proof}

\printbibliography
\end{document}